\documentclass[11pt,a4paper]{amsart}

\usepackage[english]{babel}
\usepackage[T1]{fontenc}
\usepackage[utf8]{inputenc}
\usepackage{amsmath}
\usepackage{amssymb}
\usepackage{amsthm}
\usepackage{latexsym}
\usepackage{amsfonts}
\usepackage{graphicx,color}
\usepackage{hyperref}
\usepackage{algorithm}
\usepackage[noend]{algpseudocode}

\usepackage[onehalfspacing]{setspace} 
\setdisplayskipstretch{.421} 
\newlength{\abovebis} 
\newlength{\belowbis} 
\newlength{\aboveshortbis} 
\newlength{\belowshortbis} 
\everydisplay\expandafter{%
  \the\everydisplay 
  \advance\abovedisplayskip\abovebis 
  \advance\belowdisplayskip\belowbis 
  \advance\abovedisplayshortskip\aboveshortbis 
  \advance\belowdisplayshortskip\belowshortbis 
} 

\def\R{\mathbb{R}}

\def\N{\mathbb{N}}
\def\ND{\mathcal{N}}

\def\A{\mathcal{A}}
\def\L{\mathcal{L}}

\def\Div{\textrm{div}}

\def\P{\mathcal{P}}
\allowdisplaybreaks[3]

\theoremstyle{plain}

\newtheorem{lem}{Lemma}[section]

\newtheorem{prop}[lem]{Proposition}

\theoremstyle{definition}
\newtheorem{defn}{Definition}[section]
\newtheorem{rem}{Remark}

\numberwithin{equation}{section}

\author[E. Beretta]{Elena Beretta}
\address[E. Beretta]{Dipartimento di Matematica, Politecnico di Milano, P.zza Leonardo da Vinci, 32, 20133 Milano, Italy}
\address[E. Beretta]{New York University Abu Dhabi, PO Box 129188, Saadiyat Island, Abu Dhabi, United Arab Emirates}
\email[E. Beretta]{elena.beretta@polimi.it}

\author[S. Micheletti]{Stefano Micheletti}
\address[S. Micheletti, S. Perotto, M. Santacesaria]{MOX, Dipartimento di Matematica, Politecnico di Milano, P.zza Leonardo da Vinci, 32, 20133 Milano, Italy}
\email[S. Micheletti]{stefano.micheletti@polimi.it}
\author[S. Perotto]{Simona Perotto}
\email[S. Perotto]{simona.perotto@polimi.it}
\author[M. Santacesaria]{Matteo Santacesaria}
\email[M. Santacesaria]{matteo.santacesaria@polimi.it}

\begin{document}

\title[Reconstruction via shape optimization in EIT]{Reconstruction of a piecewise constant conductivity on a polygonal partition via shape optimization in EIT}

\begin{abstract}
In this paper, we develop a shape optimization-based algorithm for the electrical impedance tomography (EIT) problem of determining a piecewise constant conductivity on a polygonal partition from boundary measurements. The key tool is to use a distributed shape derivative of a suitable cost functional with respect to movements of the partition. Numerical simulations showing the robustness and accuracy of the method are presented for simulated test cases in two dimensions. 
\end{abstract}


\maketitle

\section{Introduction}

Electrical Impedance Tomography (EIT) is a noninvasive technique, which aims to detect the conductivity inside a body from voltage and current boundary measurements.
The mathematical problem arising from EIT, known as the inverse conductivity problem, was introduced for the first time by A.-P.~Calder\'on in the early 80's \cite{Calder'on1980}. Even though it was first motivated by an application in geophysical prospecting \cite{Chambers2006}, EIT has been having big impact also in medical imaging and nondestructive testing of materials \cite{Karhunen2010}.

The conductivity problem can be stated mathematically as follows.
Consider a bounded domain $\Omega \subset \R^d$, with $d=2,3$, equipped with an electrical conductivity $\sigma \in L^{\infty}(\Omega)$ such that $\sigma(x) \geq \lambda >0$.  The corresponding Neumann-to-Dirichlet (ND) or current-to-voltage map is the operator $\ND_{\sigma}: H_0^{-1/2}(\partial \Omega) \to H_0^{1/2}(\partial \Omega)$, defined by
\begin{equation}
\ND_{\sigma}(g) = u|_{\partial \Omega},
\end{equation}
where $H^{s}_0(\partial \Omega) = \{ f \in H^s(\partial \Omega) : \int_{\partial \Omega} f\, ds = 0\}$, $g \in H_0^{-1/2}(\partial \Omega)$ and $u$ is the unique $H^1(\Omega)$-weak solution of the Neumann problem for the conductivity equation
\begin{equation}\label{schr}
\left\{
\begin{array}{ll}
-\nabla \cdot (\sigma \nabla u) &= 0, \qquad \text{in } \Omega,\\[2mm]
\sigma \displaystyle \frac{\partial u}{\partial \nu} &=g, \qquad \text{on } \partial \Omega,
\end{array}
\right.
\end{equation}
where $\nu$ is the unit outward normal to $\partial \Omega$,
satisfying the normalization condition
$$\int_{\partial \Omega} u\, ds = 0.
$$ 
The following inverse boundary value problem arises from this framework.

\vspace*{.2cm}

{\bf Inverse conductivity problem.} Given $\ND_{\sigma}$, find $\sigma$ in $\Omega$.

\vspace*{.2cm}

Since the seminal paper by A.-P.~Calder\'on, much interesting mathematics has been developed in order to address the issues of uniqueness, stability and reconstruction for this problem. Concerning uniqueness and reconstruction, we mention the breakthrough results in \cite{Sylvester1987,Novikov1988,Nachman1996}.

The conductivity problem is severely ill-posed as was noted by G.~Alessandrini in \cite{Alessandrini1988}. Despite a-priori smoothness assumptions on the unknown conductivity, a logarithmic-type continuous dependence of the conductivity on the data is the best possible one \cite{Alessandrini1988,mandache2001}. This fact makes crucial the analysis of the instability and of suitable regularization strategies in order to 
obtain successful computational reconstructions.  
Several recovery methods and procedures have been developed in the last decades.
Without being exhaustive, the possible approaches to reconstruction can be divided into two main streams:
\begin{enumerate}
\item iterative methods, based on ad-hoc regularization strategies; 
\item direct methods, where an explicit reconstruction formula of the solution is used.
\end{enumerate}
The first group includes variational-type methods, which reduce the inverse problem to a minimization problem for a least-squares constrained type functional with a suitable regularization. A pioneering paper in this direction is represented by \cite{Cheney1990a}, which applies one step of a Newton method with a constant conductivity as an initial guess. In \cite{Rondi2001} the authors introduce a Mumford-Shah type functional, in \cite{Chung2005electrical,Chan2004} a  level set representation and a total variation regularization is introduced, while in \cite{Chen1999} an augmented Lagrangian method is proposed. All these methods  are particularly suited to recover piecewise constant conductivities.

Concerning direct methods, we would like to mention the factorization method \cite{Bruhl2000, Kirsch2008}, the D-bar method \cite{Siltanen2000}, the enclosure method \cite{Ikehata2000a}, and the monotonicity method \cite{Tamburrino2002}.

Finally, statistical inversion has shed interesting insights into EIT reconstruction as well \cite{Kaipio2004a}.
 
Despite the impressive progress, there remains a big interest in developing new algorithms that take advantage of a-priori information arising from applications. 
An a-priori assumption physically relevant in many applications is to assume the conductivity to be of the form 
\begin{equation}\label{apriori}
\sigma=\sum_{j=1}^N\sigma_j\chi_{P_j},
\end{equation}
where $\mathcal{P}=\{P_j\}_{j=1}^N$ is a polygonal (polyhedral) partition of the background body $\Omega$, 
$\chi_P$ denoting the characteristic function associated with the generic region $P\subset \Omega$.
Such an assumption arises, for example, in geophysics, medical imaging, and nondestructive testing of materials, where the body under investigation contains regions, represented by the subdomains $\{P_j\}_{j=1}^N$, with different electrical properties.   

Additionally, this kind of a-priori information restores the well-posedness of the inverse problem; in particular, in the case of a given known partition, Lipschitz dependence estimates of the coefficients from the data can be shown \cite{Alessandrini2005a}.
From the results obtained for the Helmholtz equation in \cite{beretta2015stable} and \cite{beretta2016stable}, where Lipschitz stability holds, we expect that a similar result should be true also when the partition is unknown.
As shown in \cite{beretta2015stable}, a crucial role to prove Lipschitz stability  is played by the differentiability properties of the Neumann-to-Dirichlet map with respect to motions of the partition and by the derivation of an explicit formula for the derivative. This is also a crucial step towards reconstruction if we use an optimization approach to solve the inverse conductivity problem.

In this paper, we consider a cost functional, $J(\sigma)$, representing the $L^2$-norm of the difference between the potential due to the applied current and the measured potential on the boundary. This functional will be minimized in the class of conductivities $\sigma$ of the form \eqref{apriori}. 
To solve this minimization problem, we introduce an iterative gradient type method which requires the computation of the shape derivative of the functional $J(\sigma)$. This derivative has been obtained rigorously by several authors in the case of a single sufficiently smooth inclusion $\omega \subset \Omega$, i.e. for $\sigma = \sigma_1 \chi_{\omega}+\sigma_2\chi_{\Omega\backslash\omega}$ (see, for example,  \cite{hettlich1998determination,afraites2007shape,pantz2005sensibilite}).

In our case, to implement the optimization procedure, we need to differentiate $J(\sigma)$ with respect to variations of a partition. 
In \cite{beretta2017differentiability}, the authors derived, for the first time, a rigorous formula for the shape derivative of the functional $J(\sigma)$ in the planar case and for conductivities of the form
$$
\sigma=\sigma_1\chi_P+\sigma_2\chi_{\Omega\backslash P}
$$
where $P$ is a polygon strictly contained in $\Omega$. The shape derivative is expressed in terms of an integral over the boundary of $P$ and has, surprisingly, exactly the same form as the one derived in \cite{hettlich1998determination} in the case of a smooth interface, despite the presence of singularities of the gradient of the solutions to the conductivity equation at the vertices of $P$. The extension of this boundary formula seems not to be possible in the case of an arbitrary partition since the singularity of the gradient might become too strong at the vertices \cite{piccinini1972,bruce1974}.

In this paper, we follow the idea suggested in \cite{Laurain2016} and in \cite{Giacomini2017} of using a more general \textit{distributed} shape derivative of the functional expressed in terms of  an integral  over $\Omega$. The advantage of this formula is twofold: on the one hand, it allows to consider piecewise constant conductivities on very general partitions; on the other hand, it is numerically more accurate (we refer to \cite{Hiptmair2015} for a thorough comparison of the two formulas from a numerical point of view). In \cite{Laurain2016}, the authors establish the distributed shape derivative in the case of a single measurable conductivity inclusion strictly contained in $\Omega$ by using a Lagrangian approach. Here, we establish the formula computing directly and rigorously the derivative of $J$ in terms of the material derivative of the solution to a certain boundary value problem (see Lemma 2.1). The formula is valid for any partition $\mathcal{P}$ in dimension $d = 2,3$. Successively, we take advantage from this result to implement our reconstruction procedure in the two dimensional case, based on a gradient type method.

The reconstruction algorithm we present is very similar to the one introduced in \cite{Giacomini2015,Giacomini2017}. Nevertheless, there are two major differences in the implementation, which greatly affect the numerical results. The first one is a regularization step, applied at each iteration, where the number of the sides of each polygon in the partition $\mathcal{P}$ changes in order to preserve a uniform length. This considerably reduces the artifacts that typically appear in EIT reconstructions. The second major difference lies in the choice of the descent direction for the shape of the partition. While this was done by solving an additional variational problem, we propose a more direct computation which exploits the assumptions made on the conductivity.

The plan of the paper is the following one. In Section 2, we state the problem with the main assumptions, derive the shape derivative of the cost functional and show the equivalence with the boundary shape derivative established in \cite{beretta2017differentiability} for suitable two-dimensional partitions. In Section 3, we describe the reconstruction algorithm. In Section 4 we present some numerical examples which corroborate the reliability and the accuracy of the proposed approach.

\section{Mathematical framework}
\subsection{Main Assumptions} Let $\Omega \subset \R^d$, $d = 2,3$, be a bounded domain with Lipschitz boundary. Let $\P = \{P_j\}_{j=1}^{N}$ be a polytopal partition of $\Omega$, i.e., $P_j$ are open bounded polytopes (i.e., polygons in 2D, polyhedra in 3D) such that:
\begin{equation}
\displaystyle \bigcup_{j = 1}^{ N}  \overline{P_j} = \overline{\Omega}, \qquad P_j \cap P_k = \emptyset \quad \text{ for } j \neq k.
\end{equation}
\begin{defn}
Let $N \in \N$ be an integer with $N>1$, and $\lambda >0$ be a positive real number. We define the space $L^{\infty}(\Omega,N,\lambda)$ as the collection of conductivities $\sigma \in L^{\infty}(\Omega)$ such that $\sigma(x) \geq \lambda > 0$ for all $x \in \Omega$, and such that there exists a polytopal partition $\P = \{P_j\}_{j=1}^{N'}$ with $N' \leq N$, such that $\sigma$ can be written as
\begin{equation}
\sigma = \sum_{j=1}^{N'} \sigma_j \chi_{P_j}, \; \text{ with } \sigma_l \neq \sigma_m \text{ if } P_l \text{ is  adjacent to } P_m.
\end{equation}
\end{defn}
Now let $\hat \sigma, \sigma \in L^{\infty}(\Omega,N,\lambda)$. We denote by $\hat \sigma$ the unknown conductivity and by $\sigma$ a (generally) different one, which will be used in the reconstruction scheme.

Let $M$ be the number of measurements. For $1 \leq j \leq M$, let $g_j \in H_0^{-1/2}(\partial \Omega)$ be a given function representing the applied current density on $\partial \Omega$, and $f_j \in H^{1/2}(\partial \Omega)$ the corresponding measurement of the voltage on $\partial \Omega$. 
More precisely, $f_j =  \hat u_j|_{\partial \Omega}$, where $\hat u_j$ is a solution of
$$
\left\{
\begin{array}{ll}
-\nabla \cdot (\hat \sigma \nabla \hat u_j) &= 0, \qquad \text{in } \Omega,\\[2mm]
\hat \sigma \displaystyle \frac{\partial \hat u_j}{\partial \nu} &=g_j, \qquad \text{on } \partial \Omega.
\end{array}
\right.
$$
In order to recover $\hat \sigma$, we minimize the following Dirichlet least-squares fitting cost functional:
\begin{equation}\label{def:func}
J(\sigma) = \frac{1}{2} \sum_{j=1}^M\int_{\partial \Omega} \left|u_j -f_j\right|^2\, ds,
\end{equation}
for $\sigma \in L^{\infty}(\Omega,N,\lambda)$, where the state function $u_j$ solves
\begin{align}\label{eq:dir}
\left\{
\begin{array}{ll}
-\nabla \cdot (\sigma \nabla u_j) &= 0, \qquad \text{in } \Omega,\\[2mm]
\sigma \displaystyle \frac{\partial u_j}{\partial \nu} &=g_j, \qquad \text{on } \partial \Omega,
\end{array}
\right.
\end{align}
with the normalization condition
\begin{equation}
\int_{\partial \Omega} u_j \, ds= \int_{\partial \Omega} f_j\, ds.
\end{equation}

Notice that the functional $J(\sigma)$ depends on the values $\{\sigma_j\}_{j=1}^N$ and the partition $\P = \{P_j\}_{j=1}^N$, where $\sigma = \sum_{j=1}^{N} \sigma_j \chi_{P_j}$. The next subsections are devoted to the computation of the gradient of $J(\sigma)$ with respect to the above variables.

\subsection{Gradient of $J$ with respect to $\P$ -- the shape derivative}
The gradient of $J$ with respect to the partition $\P$ is actually a shape derivative. We thus want to compute the shape derivative $\langle \nabla_{\P} J, U \rangle $ of the functional $J$ at the partition $\P$ in the direction of a vector field $U = (U_1,\ldots,U_d)$. We assume that $U \in W^{1,\infty}(\R^d)$ and $U = 0$ in a neighborhood of $\partial \Omega$. This derivation has already been carried out in \cite{afraites2007shape,pantz2005sensibilite,Laurain2016} in the case of smooth inclusions, and in \cite{beretta2017differentiability} for a single polygonal inclusion. We present here a more general formula that is valid for any finite partition of a domain.

In order to compute the derivative, consider the transformation $\Phi_t( x) = x + t U( x): \R^d \to \R^d$, as smooth as $U$. We assume that $t \leq 1/(2\|U\|_{W^{1,\infty}})$ so that $\Phi_t^{-1}$ exists globally. For instance, $\Phi_t$ is a piecewise affine function that moves the nodes of the partition $\P$ (i.e., the vertices of the $P_j$'s). Note that $\Phi_t |_{\partial \Omega} = \mathrm{Id}$, $\mathrm{Id}$ denoting the identity mapping. This assumption is not a real restriction since in EIT the boundary $\partial \Omega$ is always assumed to be known and fixed. Nevertheless, it is crucial in the derivation of the main results of this section. 

The following matrix-valued functions will be useful in the following:
\begin{align}\label{def:A}
A(t) &= (D\Phi_t^{-1}) (D\Phi_t^{-1})^{T}\det(D\Phi_t),\\
\A &= \left. \frac{d A}{dt}\right|_{t=0} = \Div(U)I - (DU+DU^T),
\end{align}
where $D \Phi_t^{-1}$ and $DU$ are the Jacobian matrices of $\Phi_t^{-1}$ and $U$, respectively.
  
Consider the deformed partition $\P_t = \Phi_t(\P)$ and $\sigma_t = \sigma \circ \Phi_t^{-1}$ the corresponding conductivity. Let $G(t)$ be defined as
\begin{equation}\label{def:G}
G(t) = J(\sigma_t).
\end{equation}

Let now $u$ be any of the solutions $u_j$, with Neumann data $g$, where $g$ is any of the currents $g_j$, $j=1,\ldots,M$, and let $f$ be a boundary measurement corresponding to $g$. First, we need to study the material derivative of the solution $u$.

\begin{lem}\label{lem:mat}
The solution $u$ to problem \eqref{schr} has a material derivative $\dot u \in H^1(\Omega)$ that solves
\begin{equation}\label{eq:varfmat}
\int_\Omega \sigma \nabla \dot u \cdot \nabla w\,dx = - \int_\Omega \sigma \A \nabla u \cdot \nabla w\,dx, \qquad \forall w \in H^1(\Omega),
\end{equation}
with the normalization condition $\int_{\partial \Omega} \dot u\,ds = 0$.
\end{lem}
\begin{proof} We follow Step 1 and 2 of the proof of \cite[Theorem 3.1]{afraites2007shape}.
Let $u_t$ be the solution of
$$
\left\{
\begin{array}{ll}
-\nabla \cdot (\sigma_t \nabla u_t) &= 0, \qquad \text{in } \Omega,\\
\sigma_t \displaystyle \frac{\partial u_t}{\partial \nu} &= g, \qquad \text{on } \partial \Omega,
\end{array}
\right.
$$
with the normalization condition $\int_{\partial \Omega} u_t\,ds = \int_{\partial \Omega} f\,ds$.
Then the transported solution $\tilde u_t = u_t \circ \Phi_t$ solves the variational equation:
\begin{equation}\label{varf:utilde}
\int_\Omega \sigma A(t) \nabla \tilde u_t \cdot \nabla w\,dx -\int_{\partial \Omega}g w\,ds= 0, \qquad \forall w \in H^1(\Omega).
\end{equation}
Subtracting to \eqref{varf:utilde} the variational equation solved by $u$ and dividing by $t$, we find
\begin{equation}\label{eq:matder}
\int_\Omega \sigma A(t) \frac{\nabla \tilde u_t -\nabla u}{t} \cdot \nabla w\,dx = \int_\Omega \sigma\frac{I- A(t)}{t} \nabla u \cdot \nabla w\,dx, \quad \forall w \in H^1(\Omega).
\end{equation}
Using $\tilde u_t- u$ as test function we obtain
\begin{equation}
\frac{1}{2}\min_{x \in \Omega} \sigma(x) \left\|\frac{\nabla \tilde u_t - \nabla u}{t}\right\|_{L^2(\Omega)} \leq \left\| \frac{A(t) - I}{t} \right\|_{\infty} \| \nabla u \|_{L^2(\Omega)}.
\end{equation}
Thus, we have found that $(\tilde u_t - u)/t$ is bounded in $H^1(\Omega)$. Therefore, the sequence is weakly convergent in $H^1(\Omega)$ and its weak limit is the material derivative $\dot u$ of $u$. Passing to the limit in \eqref{eq:matder}, we find that $\dot u$ solves the desired variational formulation \eqref{eq:varfmat}. 

By the Lax-Milgram lemma, since the right-hand side is an $H^{-1}(\Omega)$ function (because of our assumptions on $U$), every solution to the variational problem \eqref{eq:varfmat} lies in $H^1(\Omega)$. In particular, the trace on $\partial \Omega$ is well defined.

Actually, we have strong convergence. Plugging $w = (\tilde u_t - u)/t$ into \eqref{eq:matder}, we obtain
\begin{equation}
\int_\Omega \sigma A(t) \nabla w \cdot \nabla w\, dx  = \int_\Omega \sigma\frac{I- A(t)}{t} \nabla  u \cdot \nabla w\, dx = B_{1,t}+B_{2,t},
\end{equation}
where 
\begin{equation}
B_{1,t} = \int_\Omega \sigma ( A(t)-I) \nabla  w \cdot \nabla w\, dx \; \text{ and } \; B_{2,t}= \int_\Omega \sigma \frac{I- A(t)}{t} \nabla  \tilde u_t \cdot \nabla w\, dx. 
\end{equation}
Thanks to the weak convergence of $(\tilde u_t - u)/t$, we obtain
\begin{equation}
B_{1,t} \to 0 \; \text{ and } \; B_{2,t} \to - \int_\Omega \sigma \A \nabla u \cdot \nabla \dot u\,dx \quad \text{as } t \to 0.
\end{equation}
Now, the variational formulation \eqref{eq:varfmat} yields $B_{2,t} \to \displaystyle \int_\Omega \sigma \nabla \dot u \cdot \nabla \dot u\,dx$.
So we have found that $\displaystyle \frac{\nabla \tilde u_t -\nabla u}{t}$ converges strongly to $\nabla \dot u$ in $L^2(\Omega)$. Using the normalization conditions for $\dot u$, $u$ and $\tilde u_t$ (which coincides with $u_t$ on $\partial \Omega$), we get the strong convergence of $\displaystyle \frac{\tilde u_t - u}{t}$ to $\dot u$ in $H^1(\Omega)$, via the Poincar\'e inequality. 
\end{proof}


Now we can derive the formula for the shape derivative of the functional $J$.
\begin{prop}
We have
\begin{equation}\label{eq:shape}
\langle \nabla_{\P} J , U \rangle = \sum_{j=1}^M \int_\Omega \sigma \A \nabla u_j \cdot \nabla z_j\,dx,
\end{equation}
where $z_j$ solves 
\begin{equation} \label{eq:adj_}
\left\{
\begin{array}{ll}
-\nabla \cdot (\sigma \nabla z_j) &= 0, \qquad \text{in } \Omega,\\[2mm]
\sigma \displaystyle \frac{\partial z_j}{\partial \nu} &= f_j-u_j, \qquad \text{on } \partial \Omega,
\end{array}
\right.
\end{equation}
with the normalization $\int_{\partial \Omega}z_j\,ds = \int_{\partial \Omega} f_j\,ds$.
\end{prop}
\begin{proof}
We need to compute the derivative $\displaystyle \frac{d G}{dt}|_{t=0}$, where $G$ is defined in \eqref{def:G}. Let $u_{j,t}, \tilde u_{j,t}$ be defined as $u_t, \tilde u_t$ in the proof of Lemma \ref{lem:mat}, with Neumann data $g_j$. By the assumption that $\Phi_t|_{\partial \Omega} = \mathrm{Id}$, we have
$$ G(t) = \frac 1 2 \sum_{j=1}^M \int_{\partial \Omega}(u_{j,t}-f_j)^2\,ds =   \frac 1 2 \sum_{j=1}^M \int_{\partial \Omega}(\tilde u_{j,t} -f_j)^2\,ds.$$
Note that
\begin{equation}\label{eq:rapg}
\frac{G(t)-G(0)}{t}=\sum_{j=1}^M \int_{\partial \Omega} \left( \frac{\tilde u_{j,t}-u_j}{t}\right)\left( \frac{\tilde u_{j,t}+u_j}{2} -f_j\right)\,ds.
\end{equation}
Using the strong convergence $\displaystyle \frac{\tilde u_{j,t}-u_j}{t} \to \dot u_j$ in $H^1(\Omega)$ and the convergence $\tilde u_{j,t} \to u_j$ in ${L^2(\partial \Omega)}$ (which follows from the fact that $\tilde u_{j,t} = u_{j,t}$ on $\partial \Omega$ and $u_{j,t} \to u_j$ in $H^1(\Omega)$), we can pass to the limit as $t \to 0$ in \eqref{eq:rapg}, thus obtaining 
\begin{align*}
\langle \nabla_{\P} J , U \rangle  = \frac{dG}{dt}|_{t=0} = \sum_{j=1}^M \int_{\partial \Omega} (u_j-f_j)\dot u_j\,ds,
\end{align*}
where $\dot u_j$ is the material derivative of $u_j$.
Using the variational formulation of the adjoint $z_j$, 
\begin{equation}\label{eq:varfadj}
\int_{\Omega} \sigma \nabla z_j \cdot \nabla w\,dx + \int_{\partial \Omega} (u_j-f_j) w\,d\sigma = 0, \qquad \forall w \in H^1(\Omega),
\end{equation}
and of the material derivative, \eqref{eq:varfmat}, we find
\begin{align*}
\langle \nabla_{\P} J , U \rangle  = - \sum_{j=1}^M \int_\Omega \sigma \nabla z_j \cdot \nabla \dot u_j\,dx = \sum_{j=1}^M \int_\Omega \sigma \A \nabla u_j \cdot \nabla z_j\,dx, 
\end{align*}
which is the desired formula.
\end{proof}\smallskip

We now recall also the gradient of $J$ with respect to $\{ \sigma_j \}_{j=1}^N$, that will be used in the reconstruction algorithm:

\begin{equation}\label{eq:gradc}
\frac{d J}{d \sigma_j} = \sum_{k=1}^M \int_{P_j} \nabla u_k \cdot \nabla z_k\,dx, \quad j = 1,\ldots,N,
\end{equation}
where $z_k$ solves \eqref{eq:adj_} with a normalization (e.g., $\displaystyle \int_{\partial \Omega}z_k\,ds = \int_{\partial \Omega} f_k\,ds$).

\subsection{Equivalence of the distributed and of the boundary integral formulas in the case of a single polygonal inclusion}
We emphasize that, in the case of a single poly\-gonal inclusion, the shape derivative of the functional $J(\sigma)$ has been computed rigorously in \cite{beretta2017differentiability}, and expressed as an integral on the boundary of such an inclusion.  It is unclear if this boundary representation of the shape derivative is still valid in the case of an arbitrary partition.

Let $P$ be a polygon strictly contained in $\Omega$ and let
\[
u^+=u|_{\Omega\backslash \bar P},\,\,\,\, z^+=z|_{\Omega\backslash \bar P}.
\]
and 
\[
u^-=u|_{ P},\,\,\,\, z^-=z|_{P}.
\]

\begin{prop}
Assume $\Omega=P\cup(\Omega\backslash P)$, with $P$ a polygon strictly contained in $\Omega$, and let $\sigma|_{P} = k$ and $\sigma = 1$ outside $P$.  Then, we have the following equivalent formulas for the shape derivative of the functional $J$ at $P$ in the direction given by $U$:
\begin{align*}
\left.\frac{d G}{dt}\right|_{t=0} &=\sum_{j=1}^M \int_{\Omega} \sigma( x) \A \nabla u_j( x) \cdot \nabla z_j(x)\,dx \\
& = (k-1) \sum_{j=1}^M\int_{\partial P}\left(\frac{1}{k}\frac{\partial u_j^+}{\partial \nu}\frac{\partial z^+_j}{\partial \nu}+\nabla_{\tau}u_j \cdot\nabla_{\tau}z_j\right)U_{\nu}\,ds, 
\end{align*}
where $\nu$ is the unit outward normal vector to $\partial P$, $\nabla_{\tau}$ is the tangential gradient, $U_{\nu} = U\cdot \nu$, functions $u_j,z_j$ satisfy \eqref{eq:dir} and \eqref{eq:adj_}, respectively.
\end{prop}
\begin{proof}
We can assume, without loss of generality, $M=1$ and denote by $u$ and $z$  the solutions corresponding to the datum $f$. Let $B_\varepsilon$ be the union of balls of radius $\varepsilon$ centered at each vertex of $P$,
and let us consider the splitting 
\[
\int_{\Omega} \sigma \A \nabla u \cdot \nabla z\,dx=\int_{\Omega\backslash B_{\varepsilon}} \sigma \A \nabla u \cdot \nabla z\,dx+\int_{B_{\varepsilon}} \sigma \A \nabla u \cdot \nabla z\,dx.
\]
Denote by 
\[
\Omega_{\varepsilon}^+=\Omega\backslash\overline{P\cup B_{\varepsilon}},\,\,\, \Omega_{\varepsilon}^-=P\backslash\overline{ B_{\varepsilon}},
\]
and observe that, by standard regularity results (see, for example, \cite{Ito2008}), it can be shown that $u^+,z^+\in H^2(\Omega_{\varepsilon}^+)$, and  $u^-,z^-\in H^2(\Omega_{\varepsilon}^-)$. Then,  applying Green's formula in $\Omega_{\varepsilon}^{\pm}$, observing that $U$ has compact support in $\Omega$, and using the following identity
\begin{align*}
\A \nabla u \cdot \nabla z &= -\Div(b) + (U\cdot \nabla u) \Delta z +(U\cdot \nabla z) \Delta u\\
&=-\Div(b) \qquad \text{in } \Omega_{\varepsilon}^+\cup\Omega_{\varepsilon}^- =\Omega\backslash \bar B_{\varepsilon},
\end{align*}
where 
$$b = (U\cdot \nabla u) \nabla z + (U\cdot \nabla z) \nabla u- (\nabla u \cdot \nabla z) U,$$ 
we easily derive
$$
\int_{\Omega\backslash \bar B_{\varepsilon}} \sigma \A \nabla u \cdot \nabla z\,dx=\int_{\partial P\backslash \bar B_{\varepsilon}}[\sigma b]\cdot\nu\,ds+\int_{\partial B_{\varepsilon}}\sigma b\cdot\nu\,ds,
$$
where we use notation $[f] = f^+-f^-$ to denote the jump off across $\partial P$. Using the transmission conditions satisfied by $u$ and $z$ across $\partial P$, we end up with the following relation
\begin{align} \nonumber
\int_{\Omega\backslash \bar B_{\varepsilon}} &\sigma \A \nabla u \cdot \nabla z\,dx = 
\quad (k-1) \int_{\partial P}\left(\frac{1}{k}\frac{\partial u^+}{\partial \nu}\frac{\partial z^+}{\partial \nu}+\nabla_{\tau}u \cdot\nabla_{\tau}z\right)U_{\nu}\,ds\\ \nonumber
&-(k-1) \int_{\partial P\cap \bar B_{\varepsilon}}\left(\frac{1}{k}\frac{\partial u^+}{\partial \nu}\frac{\partial z^+}{\partial \nu}+\nabla_{\tau}u \cdot\nabla_{\tau}z\right)U_{\nu}\,ds+\int_{\partial B_{\varepsilon}}\sigma b\cdot\nu\,ds,
\end{align}
where $U_{\nu}=U\cdot \nu$.
Hence,
\begin{align*} \nonumber
\int_{\Omega} &\sigma \A \nabla u \cdot \nabla z\,dx = 
\quad (k-1) \int_{\partial P}\left(\frac{1}{k}\frac{\partial u^+}{\partial \nu}\frac{\partial z^+}{\partial \nu}+\nabla_{\tau}u \cdot\nabla_{\tau}z\right)U_{\nu}\,ds\\ 
&\underbrace{-(k-1) \int_{\partial P\cap \bar B_{\varepsilon}}\left(\frac{1}{k}\frac{\partial u^+}{\partial \nu}\frac{\partial z^+}{\partial \nu}+\nabla_{\tau}u \cdot\nabla_{\tau}z\right)U_{\nu}\,ds}_{R_1}\\
&+\underbrace{\int_{\partial B_{\varepsilon}}\sigma b\cdot\nu\,ds}_{R_2}+\underbrace{\int_{B_{\varepsilon}} \sigma \A \nabla u \cdot \nabla z\,dx}_{R_3}\\
&=(k-1) \int_{\partial P}\left(\frac{1}{k}\frac{\partial u^+}{\partial \nu}\frac{\partial z^+}{\partial \nu}+\nabla_{\tau}u \cdot\nabla_{\tau}z\right)U_{\nu}\,ds+R_1+R_2+R_3.
\end{align*}
From \cite{bellout1992stability}, we have the following upper bounds of the gradients of $u$ and $z$ in a neighbourhood of the vertices:
\begin{equation}\label{eq:gradest}
|\nabla u (x)|\leq C |x-\tilde x|^{\alpha-1}, \quad |\nabla z (x)| \leq C |x-\tilde x|^{\alpha-1}, 
\end{equation}
for some constant $C>0$, $\alpha > 1/2$ and $x$ sufficiently close to a vertex $\tilde x$, with $x\neq \tilde x$.  This, jointly with the regularity assumptions on the vector field $U$, implies that 
\begin{align*} \nonumber
\int_{\Omega} &\sigma \A \nabla u \cdot \nabla z\,dx = 
(k-1) \int_{\partial P}\left(\frac{1}{k}\frac{\partial u^+}{\partial \nu}\frac{\partial z^+}{\partial \nu}+\nabla_{\tau}u \cdot\nabla_{\tau}z\right)U_{\nu}\,ds+O(\epsilon^{2\alpha-1}),
\end{align*}
and, since $\alpha>1/2$, letting $\epsilon\rightarrow 0$ in the last equation, we finally obtain
\[
\int_{\Omega} \sigma \A \nabla u \cdot \nabla z\,dx=(k-1) \int_{\partial P}\left(\frac{1}{k}\frac{\partial u^+}{\partial \nu}\frac{\partial z^+}{\partial \nu}+\nabla_{\tau}u \cdot\nabla_{\tau}z\right)U_{\nu}\,ds,
\]
which ends the proof.
\end{proof}

\begin{rem}
It is straightforward to check that the boundary formula derived in \cite{beretta2017differentiability} (and hence also the last proposition) extends to the case of a finite number of well separated, polygonal inclusions at a positive distance to the boundary of $\Omega$.
\end{rem}

\begin{rem}
The formula for the shape derivative of the functional $J$ can also be obtained following the Lagrangian approach as in \cite{Laurain2016}. 
In fact, defining the following Lagrangian 
\begin{align} \nonumber 
\tilde \L (u,z,t,U) &= \frac{1}{2}\int_{\partial \Omega} \left|u -f\right|^2\,d\sigma+ \int_{\Omega} \sigma A(t)  \nabla u \cdot \nabla z\,dx - \int_{\partial \Omega} g z\,d\sigma,
\end{align}
where $A(t)$ is defined as in \eqref{def:A}, while functions $u$ and $z$ satisfy \eqref{eq:dir} and \eqref{eq:adj_}, respectively corresponding to a given Neumann datum $g$, 
it is possible to show that 
\[
\frac{dG}{dt}|_{t=0}= \frac{\partial}{\partial t} \tilde \L(u,z,t,U)|_{t=0}= \int_{\Omega} \sigma(x) \A \nabla u(x) \cdot \nabla z(x)\,dx.
\]
\end{rem}

\section{Reconstruction algorithm}
We use the descent gradient method to solve the minimization problem involving the functional $J$ in \eqref{def:func}. Although the algorithm holds in dimension 2 and 3, we focus on the planar case, since computational experiments are in 2D.	

The algorithm adopted in this paper differs from reconstruction algorithms available in the literature (\cite{Giacomini2015,Giacomini2017})
because of two major features, which are separately addressed in the next sections.

\subsection{Regularization} The first novelty adopted in the reconstruction algorithm consists in 
increasing or decreasing the number of vertices of the polygons in the partition in order to impose some \textit{regularity} on the reconstruction. This simple trick considerably improves the reconstruction quality, smoothing out artifacts and irregularities that often characterize EIT reconstructions (see Section 4 below for practical examples).

\subsection{Construction of the descent direction} The second distinguishing feature of the adopted reconstruction algorithm 
concerns the computation of the descent direction associated with the partition vertices. A standard approach consists in solving a discretized version of the equation
\begin{equation}\label{eq:grad}
\langle \theta^k, \delta \theta \rangle + \langle  \nabla_{\P} J(\sigma^k), \delta \theta \rangle = 0 \qquad \text{for every } \delta \theta \in X,
\end{equation}
where $\theta^k \in X$ is the descent direction at iteration $k$, $\sigma^k$ is the conductivity obtained at the iteration $k$ (see Section \ref{sec:recalg} for a precise definition), while $X \subset W^{1,\infty}(\Omega,\R^2)$. There are several approaches to discretize and solve equation \eqref{eq:grad}. We refer to \cite{Giacomini2017,Dogan2007} for a thorough discussion on this subject. Here we propose a different, more straightforward, derivation. Since the partition is defined by its vertices, we compute the descent direction for each vertex individually, and then we update the partition. The formula we use at iteration $k$ to define the descent direction, $\theta_l^k \in \R^2$, for a given vertex, $V_l^k$, $l=1,\ldots,N^k_V$, with $N^k_V$ the number of the partition vertices at the iteration $k$, is:
\begin{equation} \label{eq:desc}
\theta_l^k = -\left(\langle \nabla_{\P} J(\sigma^k), U^k_{l,1} \rangle, \langle \nabla_{\P} J(\sigma^k), U^k_{l,2} \rangle \right),
\end{equation}
where the vector fields $U^k_{l,1},U^k_{l,2}$ are chosen as follows:
\begin{itemize}
\item[(1)] $U^k_{l,1},U^k_{l,2} \in W^{1,\infty}(\Omega,\R^2)$ with  support strictly contained in $\Omega$
\item[(2)] $U^k_{l,1},U^k_{l,2}$ piecewise linear on the edges of the partition and such that
\begin{equation}\label{eq:hat}
U^k_{l,1}(V_j^k) = (\delta_{jl},0), \qquad U^k_{l,2}(V_j^k) = (0,\delta_{jl}),
\end{equation}
\end{itemize}
where $\delta_{jl}$ indicates the Kronecker delta. 
With this choice, the vector 
$$\left(\langle \nabla_{\P} J(\sigma^k), U^k_{l,1} \rangle, \langle \nabla_{\P} J(\sigma^k), U^k_{l,2} \rangle \right)$$ 
represents, at the iteration $k$, the gradient of the functional $J$ with respect to the position of a single vertex, $V_l^k$. 
In practice, in view of a finite element discretization, we choose the vector fields 
$U^k_{l,1}$ and $U^k_{l,2}$ as 
\begin{equation}
U^k_{l,1} = (\varphi^k_l,0), \qquad U^k_{l,2} = (0,\varphi^k_l),
\end{equation}
where $\varphi^k_l$ is the \textit{hat function} associated with the node $V^k_l$ of a coarse mesh that contains both the vertices and the edges of the partition $\P$ but with no additional nodes on the sides of the polygons. More precisely, $\phi^k_l (V^k_j) = \delta_{jl}$, and it is piecewise linear. Examples of coarse meshes adapted to a partition can be seen in every plot in Section 4. Thus, any other hat function corresponding to such a coarse mesh would produce the same results, since the shape derivative is supported on the edges of the partition.

We remark that we need to assume the knowledge of the number $N$ of polygons in the partition, since the algorithm is unable to make changes in the topology. A possible approach to overcome this limitation, using topological derivatives and a level set formulation, has been studied in \cite{Hintermuller2008}.\smallskip

\subsection{The reconstruction algorithm}\label{sec:recalg}
Let $N^k_V$ be the number of vertices of the partition $\P^k = \{ P^k_j\}_{j=1}^N$ at iteration $k$, and $\{ V^k_l \}_{l =1}^{N^k_V}$ be the corresponding set of vertices. Let $N^k_j$ be the number of vertices of the polygon $P^k_j$, for $j=1,\ldots,N$. At each iteration, we consider the conductivity $\sigma^k = \sum_{j=1}^N \sigma_j^k \chi_{P^k_j}$. In the algorithm we introduce the parameters $\delta_1, \delta_2 >0$ as threshold for the distances between consecutive vertices in the regularization step and a tolerance $tol >0$ is chosen to control the size of the shape gradient as a stopping criterion. Finally, the step sizes $\alpha^j_k > 0$ and $\beta_k >0$ can be fixed or obtained by line search. However, in all numerical experiments  below, we keep them fixed.

\begin{algorithm}[H]
\caption{Reconstruction algorithm}
Choose an initial guess $\{\sigma^0_j\}_{j=1}^N$, $\{V^0_l\}_{l=1}^{N^0_V}$; set $k = 0$ and iterate:
\begin{algorithmic}[1]
\State For all consecutive vertices $V_l^k,V_{l+1}^k$ in each polygon:
\State \qquad if $\|V_l^k-V_{l+1}^k\| < \delta_1$ then remove $V_j^k$;
\State \qquad  if $\|V_l^k-V_{l+1}^k\| > \delta_2$ then introduce the new vertex, $\frac{1}{2}(V_l^k + V_{l+1}^k)$;
\State Generate a coarse mesh corresponding to this new set of vertices identifying the partition $\P^k$, and the corresponding conductivity $\sigma^k$;
\State Compute the solutions of the state \eqref{eq:dir} and adjoint \eqref{eq:adj_} problem on a refined mesh;
\State Compute the gradient w.r.t. the coefficients $\frac{d J}{d \sigma_j}(\sigma^k)$, for $j = 1,\ldots,N$ via \eqref{eq:gradc};
\State Compute the descent directions $\theta_l^k$ via \eqref{eq:desc}, corresponding to each vertex, $V_l^k$, $l = 1,\ldots,N^k_V$;
\State Update the coefficients and the partition: $\sigma_j^{k+1} = \sigma_j^k -\alpha^j_k \frac{d J}{d \sigma_j}(\sigma^k)$, $V^{k+1}_l = V^{k}_l+ \beta_k \theta^k_l$;
\State If $\max_{l = 1,\ldots,N^k_V} \|\theta^k_l\| > tol$, set $k = k+1$ and repeat.
\end{algorithmic}
\end{algorithm}

The stopping criteria for the algorithm may be improved following the error analysis in \cite{Giacomini2017}. For the sake of simplicity, in the present work, we decided to focus on the new shape updates in the algorithm, neglecting other issues already investigated in earlier works. The 
inclusion of an advanced stopping criterion as well as of a robust line search algorithm for the step sizes constitute possible topics for a future investigation.

\section{Numerical tests}
In this section, we show the effectiveness of the proposed reconstruction scheme. A series of numerical experiments have been carried out in order to assess numerically that the algorithm enable us to recover simultaneously the partition and the values of the conductivity for some significant configurations.

For the sake of simplicity, we consider $\Omega = (0,1)^2$. For each test, we create a mesh adapted to the unknown conductivity that is used to provide the boundary data. For the reconstruction, a coarse mesh adapted to the approximate partition is generated at each iteration. A refined mesh is also constructed in order to compute the state and the adjoint problem.

\begin{figure}
\begin{picture}(300,100)
\put(-40,-10){\includegraphics[width=4cm]{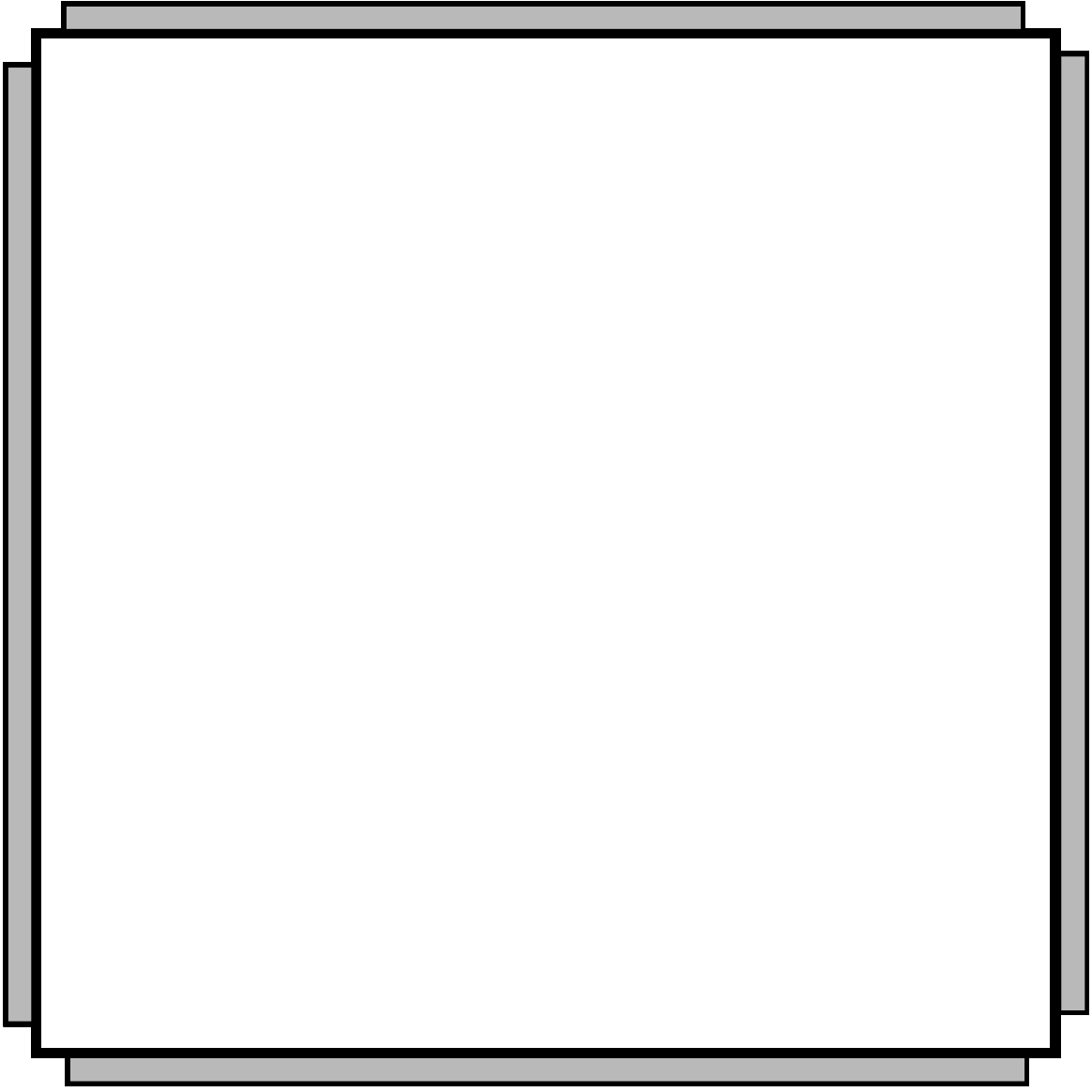}}
\put(100,-10){\includegraphics[width=4cm]{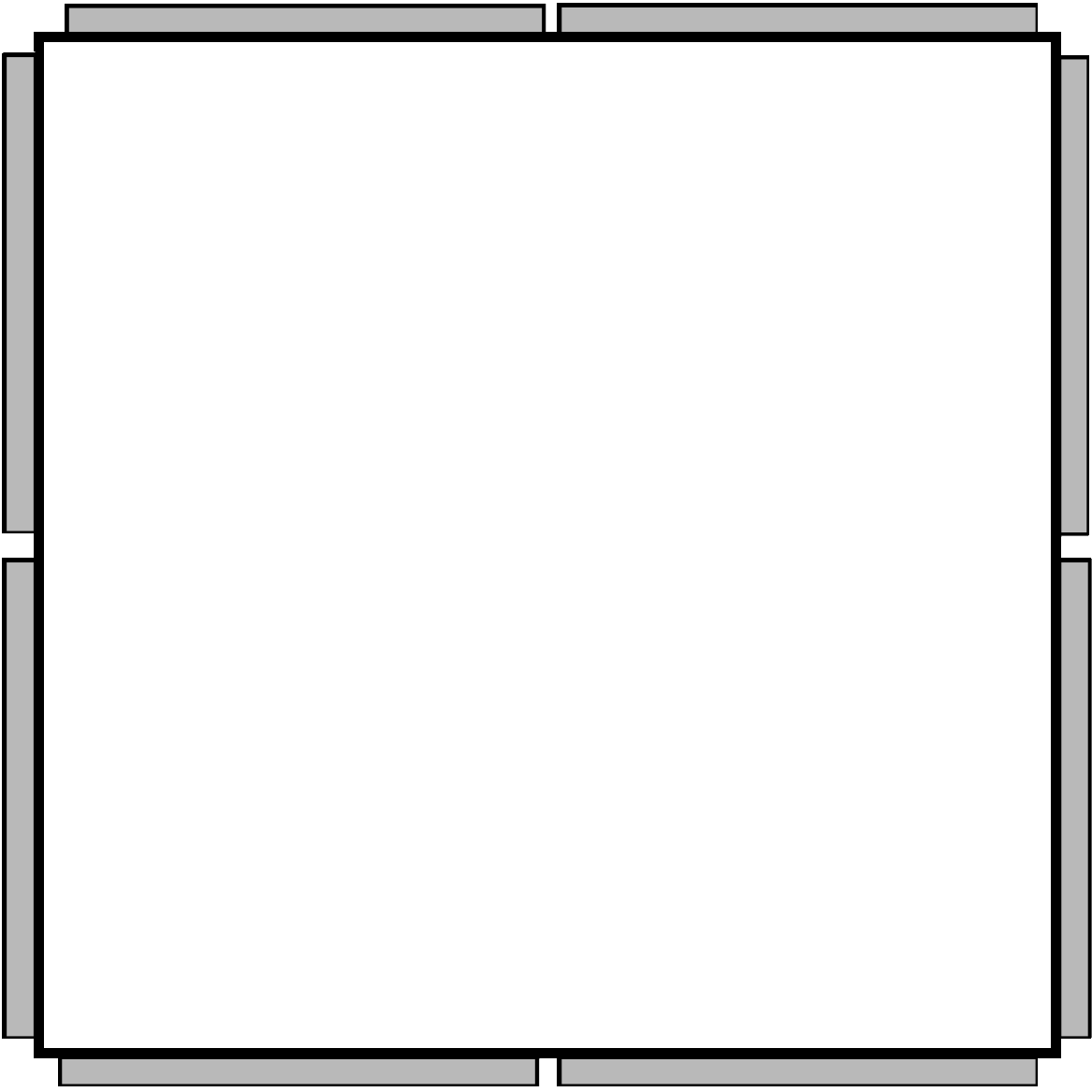}}
\put(240,-10){\includegraphics[width=4cm]{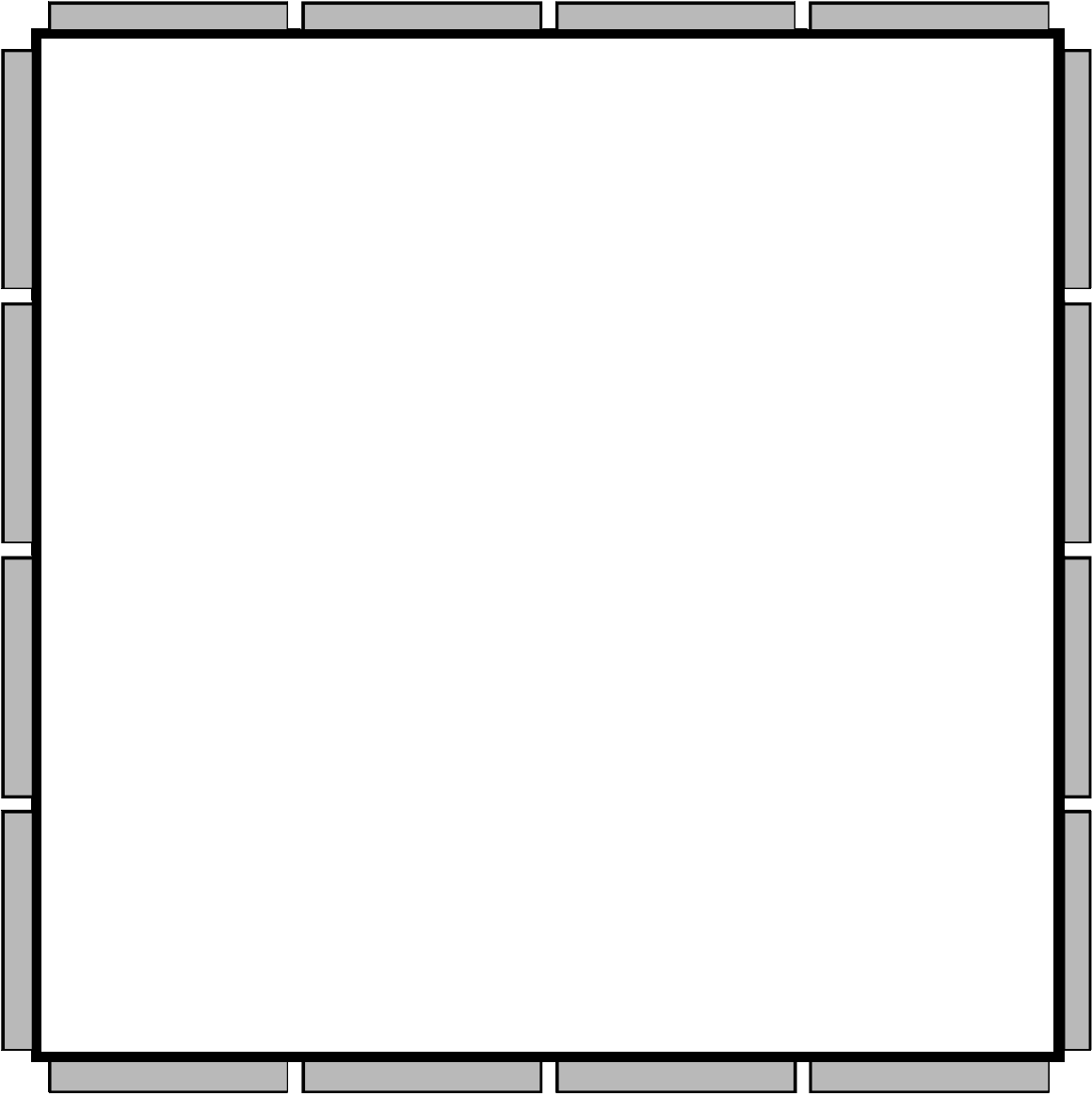}}
\end{picture}
\caption{\label{fig:elet} From left to right: electrodes (in gray) corresponding to the three data sets of 6, 28 and 120 boundary measurements, respectively.}
\end{figure}

The boundary data are generated in the following manner. We set $\sigma \frac{\partial u}{\partial \nu} = 1$ on one of the four sides of $\Omega$, $-1$ on another side and $0$ elsewhere. In this way, we obtain $6$ independent current patterns, corresponding to $4$ electrodes, one for each side. Then, we divide each side of $\Omega$ in half. We set $\sigma \frac{\partial u}{\partial \nu} = 1$ on one half, $-1$ on another half side and $0$ elsewhere. This gives $28$ independent current patterns, corresponding to $8$ electrodes. We iterate this procedure one more time. In this way, we construct sets of 6, 28 and 120 boundary data, corresponding to 4, 8 and 16 electrodes, respectively. See Figure \ref{fig:elet} for a scheme of the electrode position corresponding to the data sets.

We also add a uniform noise to the data. More precisely, given a noiseless boundary measurement $f_j \in H^{1/2}(\partial \Omega)$, $j=1, \ldots,M$, the noisy data $\tilde f_j$ is obtained by adding to $f_j$ a uniform noise in the following way:
\begin{equation}\notag
\tilde f_j (x) = f_j (x) +  \varepsilon \|f\|_{L^2(\partial \Omega)},
\end{equation}
where $x \in \partial \Omega$ is a boundary vertex of the mesh that generated $f_j$ and $\varepsilon$ is a uniform random real in $(-\gamma,\gamma)$, where $\gamma> 0$ is chosen according to the noise level.
To measure the noise level, we use the relative error on the boundary in the $L^2$-norm, that is the following quantity:
\begin{equation}\notag
\frac{ \sqrt{\sum_{j=1}^M \|\tilde f_j - f_j\|^2_{L^2(\partial \Omega)} }}{\sqrt{\sum_{j=1}^M \| f_j\|^2_{L^2(\partial \Omega)} }}.
\end{equation}

The regularization parameters, $\delta_1$ and $\delta_2$, are chosen experimentally. Since the initial guess is always a regular polygon (or a collection of regular polygons), we choose $\delta_1 = \alpha_1 \delta$, $\delta_2 = \alpha_2 \delta$, where $\delta$ is the length of the side of the initial guess, with $\alpha_1 < 1$, $\alpha_2 > 1.5$. This is done in order to have, at each iteration, a partition with edges of similar length.

All the computations are performed using FreeFem++ \cite{Hecht2012}.

\subsection{Shape reconstruction}

\begin{figure}
\begin{picture}(300,420)
\put(-60,220){\includegraphics[width=7cm]{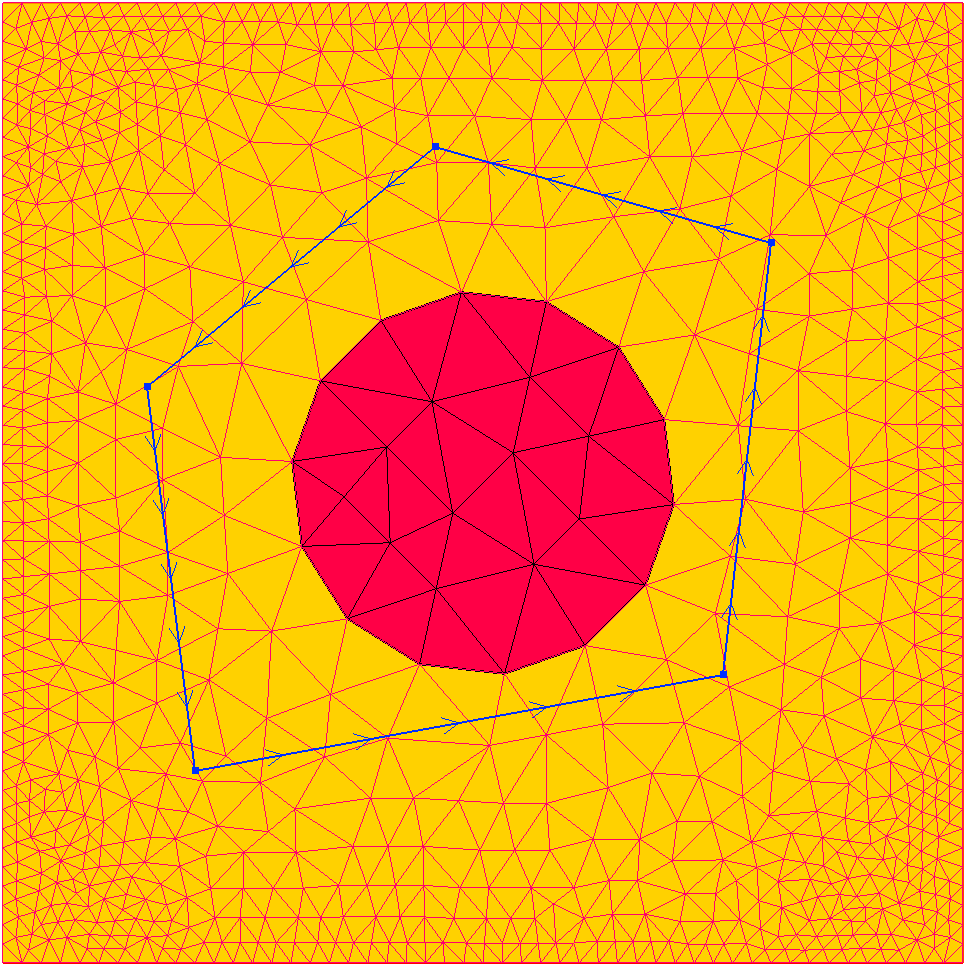}}
\put(170,220){\includegraphics[width=7cm]{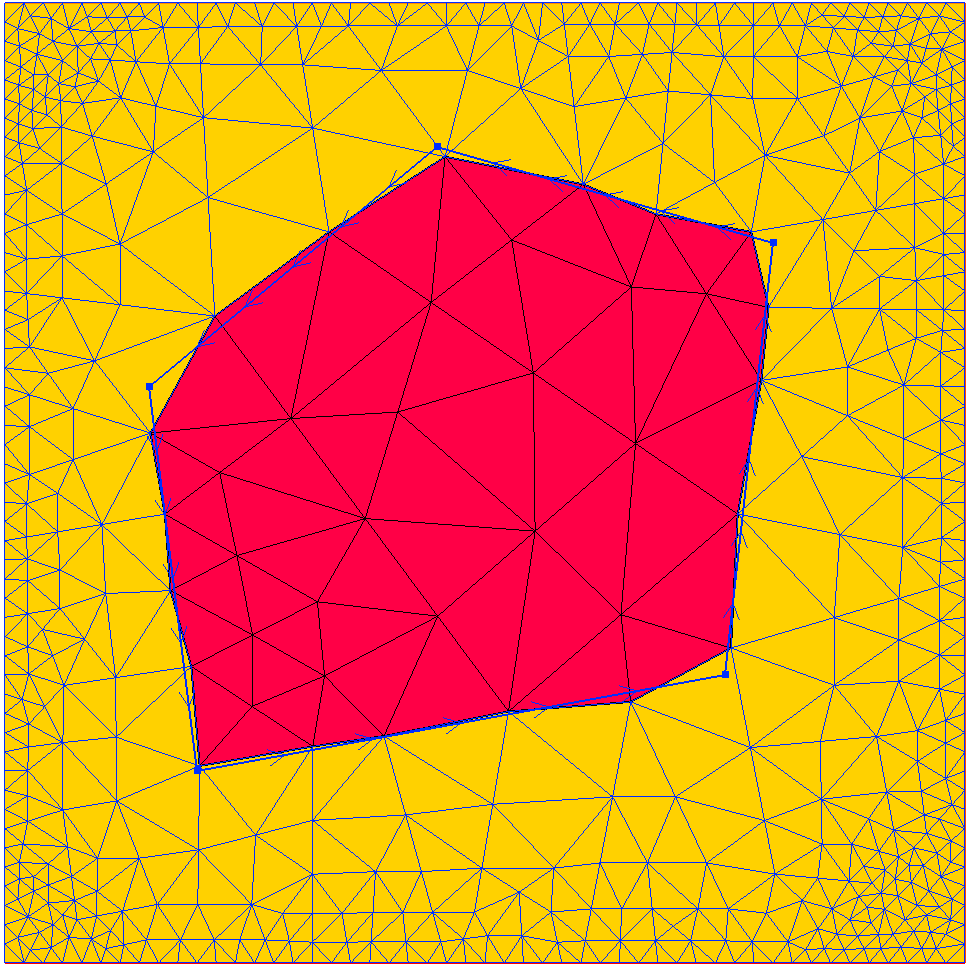}}
\put(-60,0){\includegraphics[width=7cm]{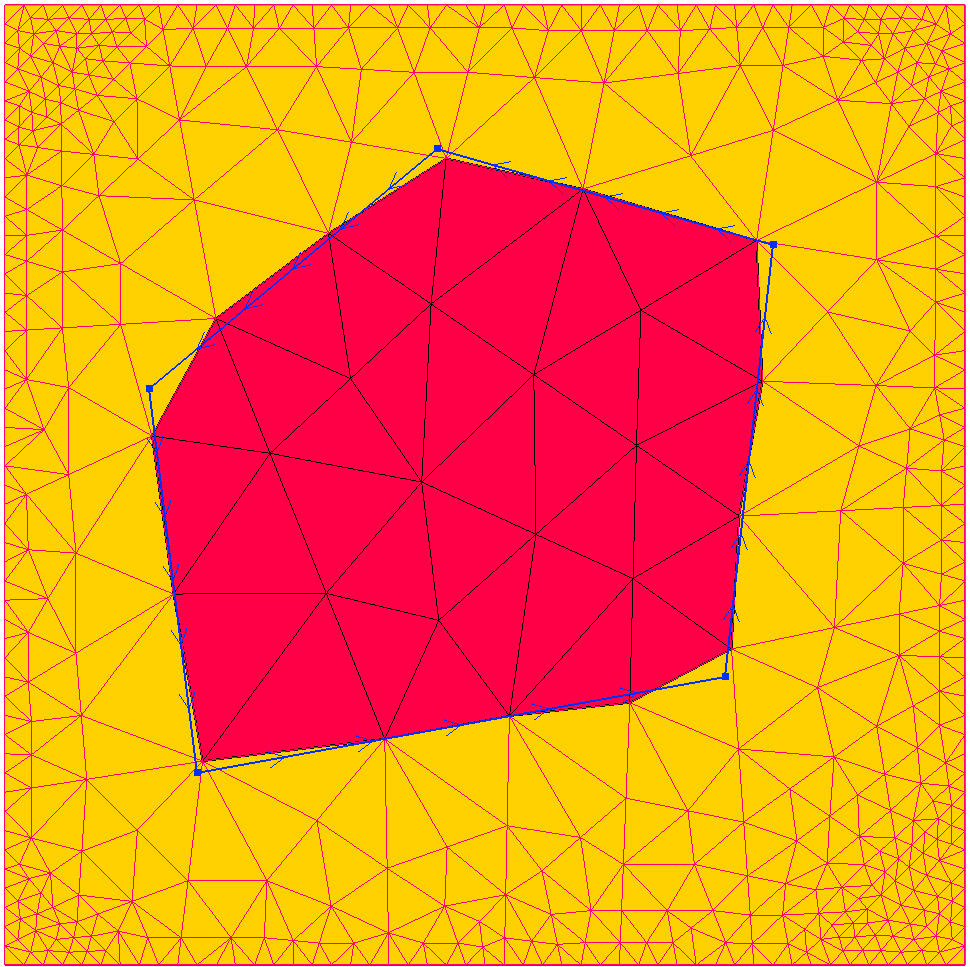}}
\put(170,0){\includegraphics[width=7cm]{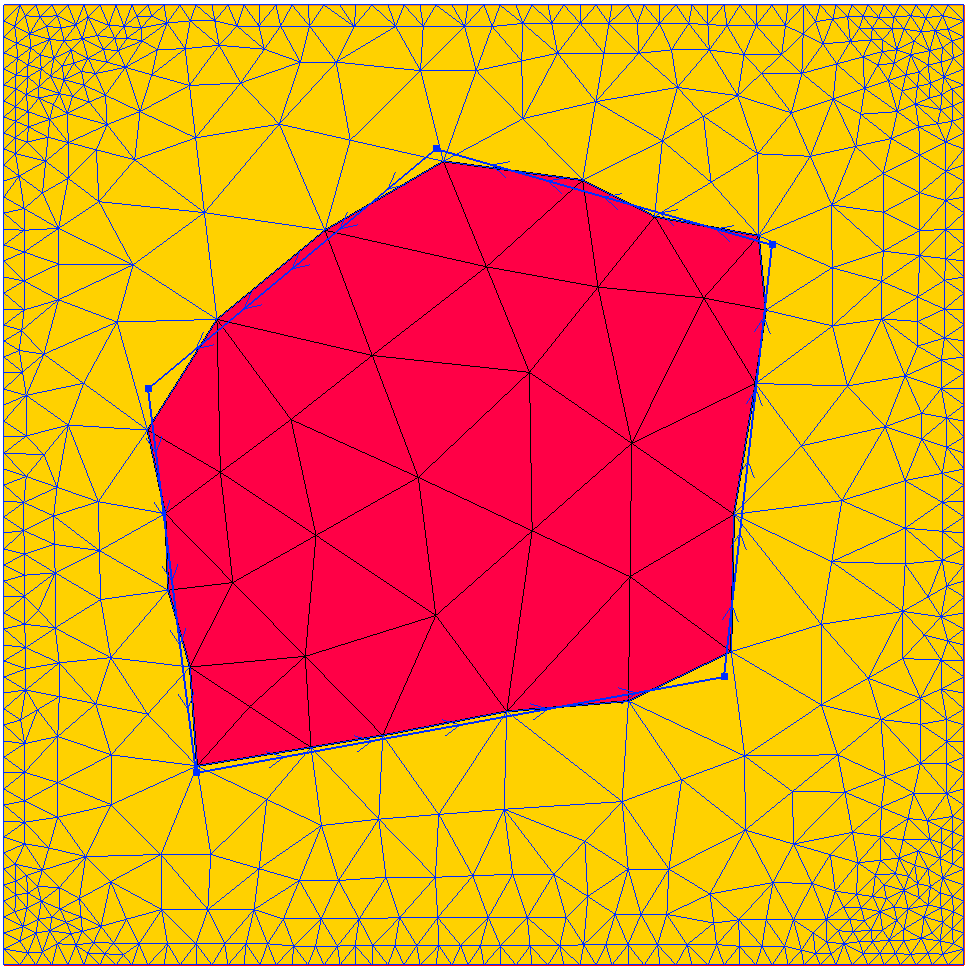}}
\end{picture}
\caption{\label{fig:poly_6}Reconstruction of a piecewise constant conductivity on an asymmetric pentagon starting from 6 measurements. The blue line represents the target shape. The values of the conductivity are known. Initial guess (top-left). Reconstruction from noiseless data with (top-right) and without (bottom-left) the regularization step. Bottom Right: reconstruction with regularization for 3\% noisy data.}
\end{figure}
\begin{figure}
\begin{picture}(300,420)
\put(-60,220){\includegraphics[width=7cm]{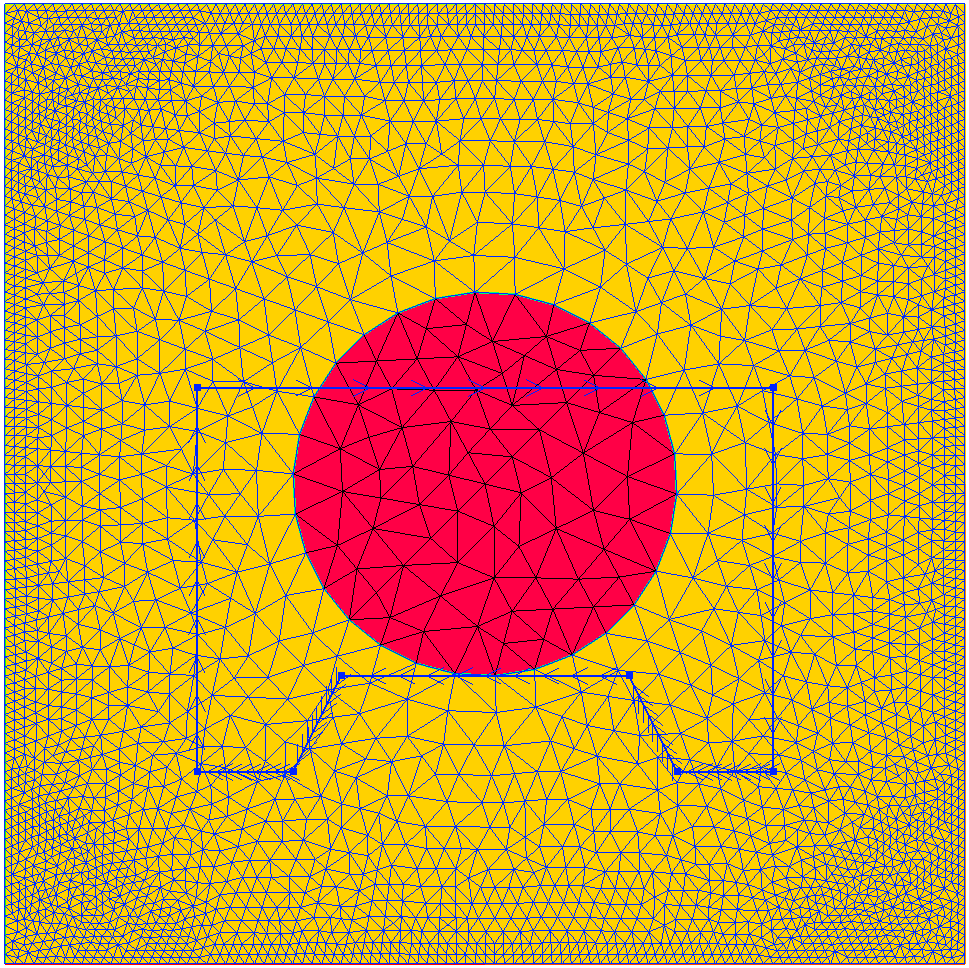}}
\put(170,220){\includegraphics[width=7cm]{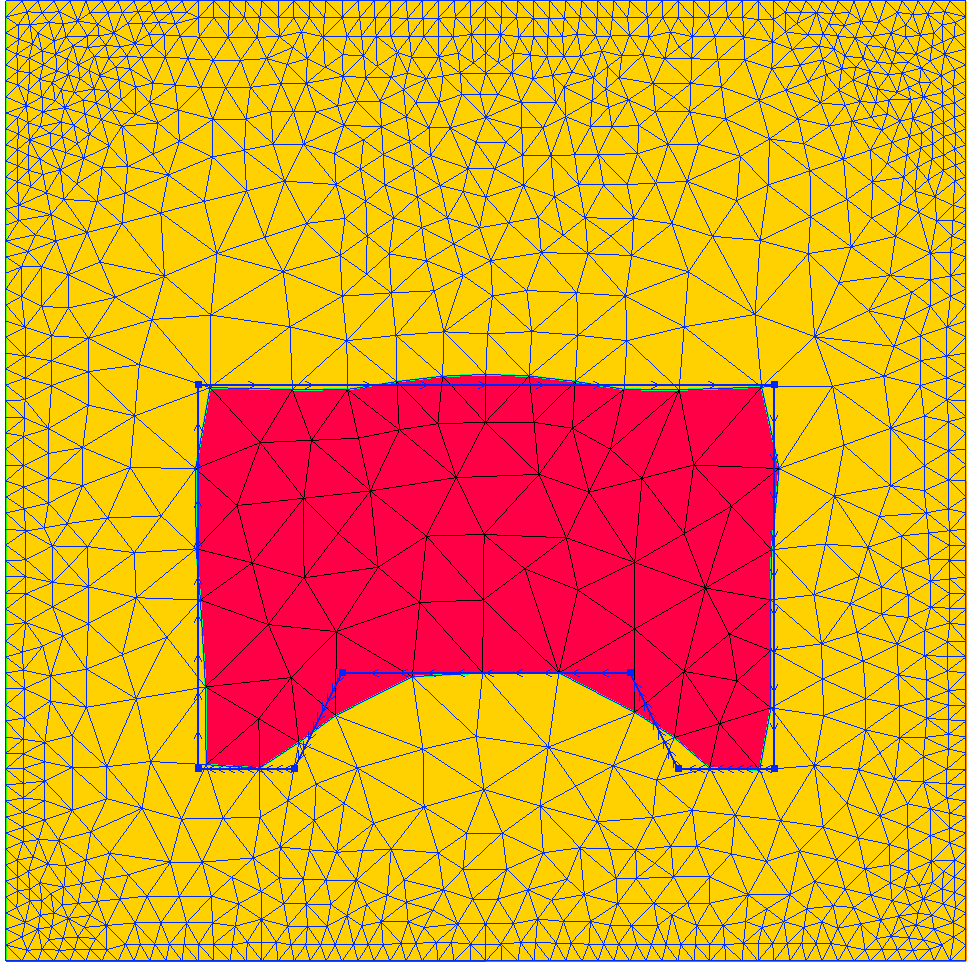}}
\put(-60,0){\includegraphics[width=7cm]{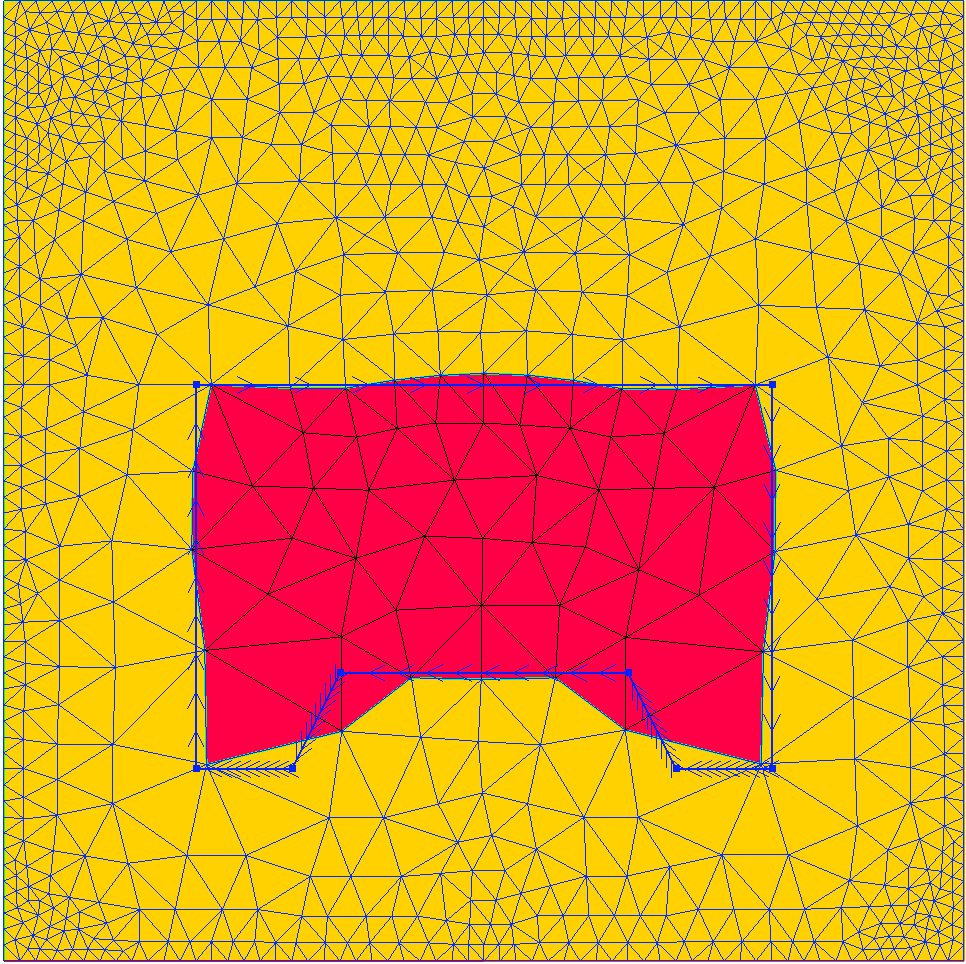}}
\put(170,0){\includegraphics[width=7cm]{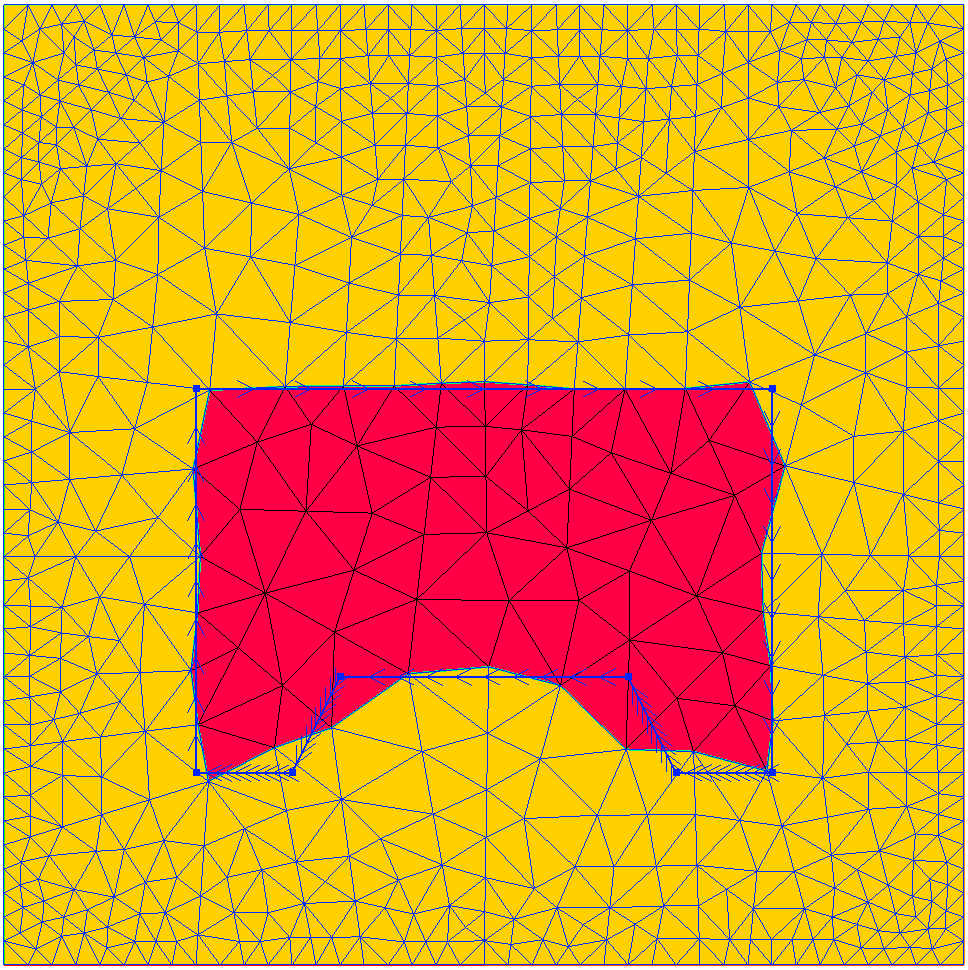}}
\end{picture}
\caption{\label{fig:cav_28}Reconstruction of a piecewise constant conductivity on a non-convex polygon starting from 28 measurements. The blue line represents the target shape. The values of the conductivity are known.  Initial guess (top-left). Reconstruction from noiseless data with (top-right) and without (bottom-left) the regularization step. Bottom Right: reconstruction with regularization for 3\% noisy data.}
\end{figure}
\begin{figure}
\begin{picture}(300,420)
\put(-60,220){\includegraphics[width=7cm]{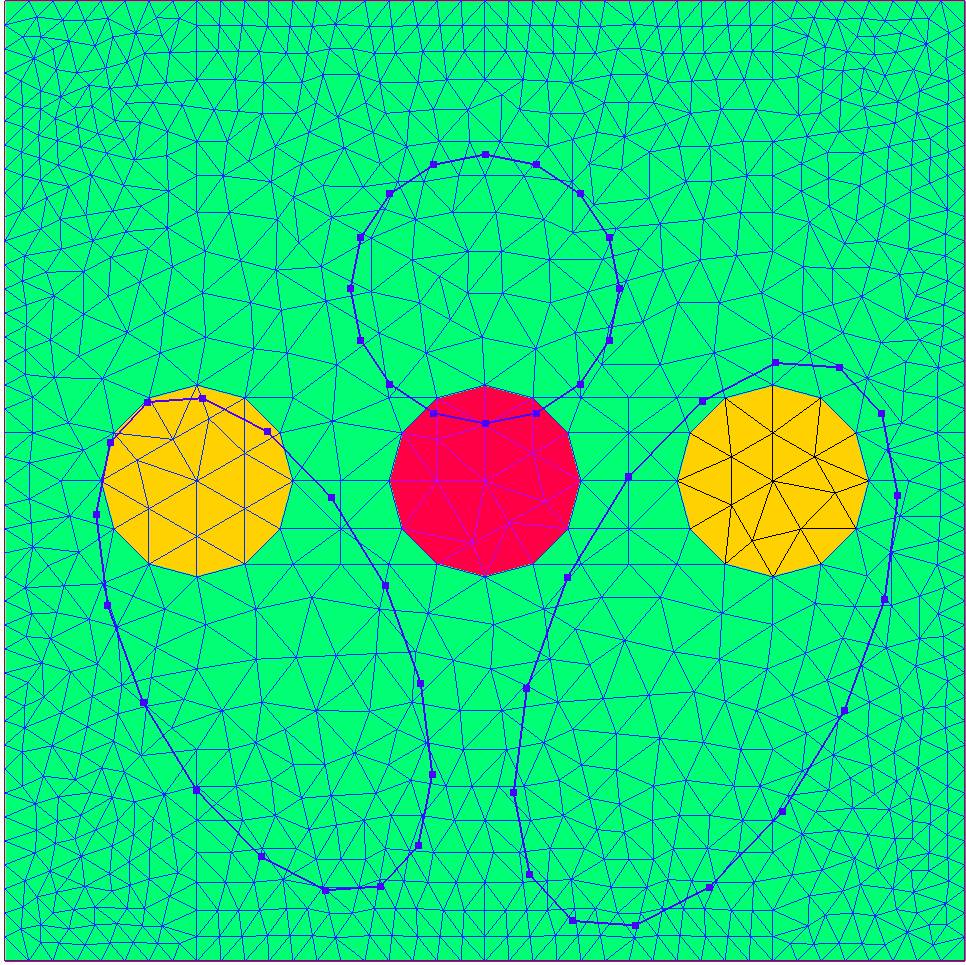}}
\put(170,220){\includegraphics[width=7cm]{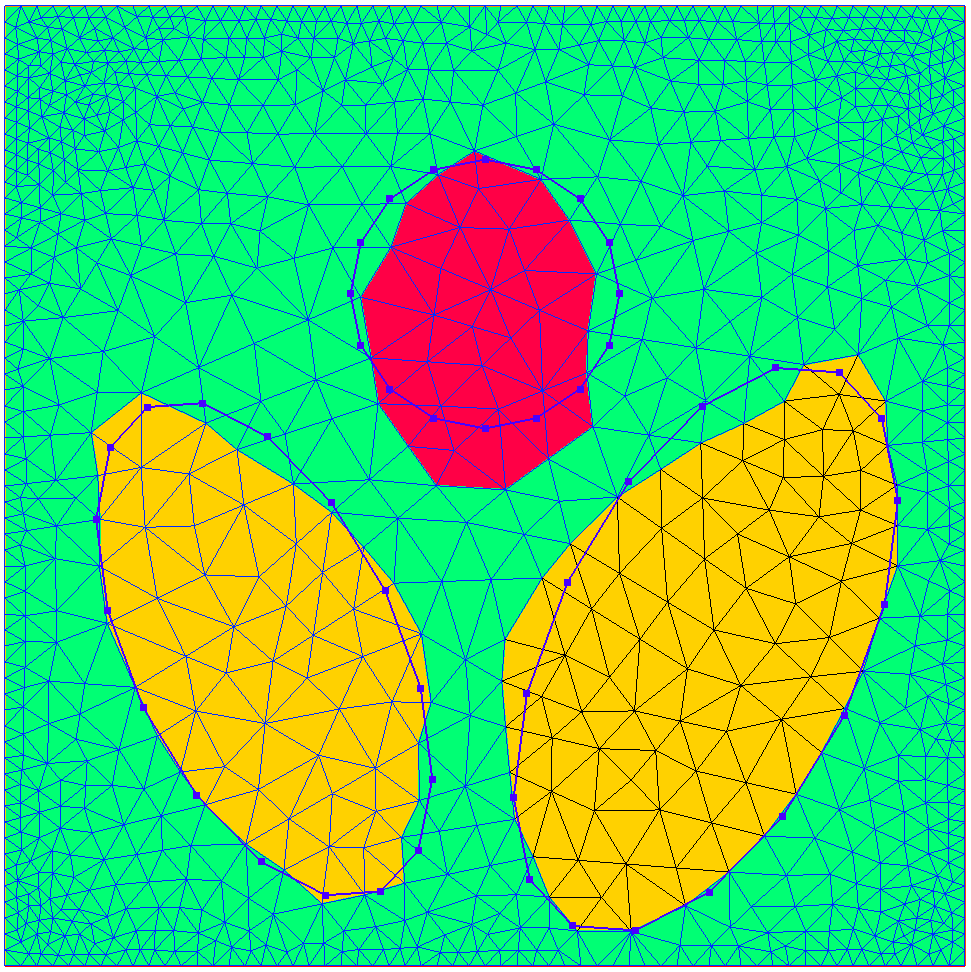}}
\put(-60,0){\includegraphics[width=7cm]{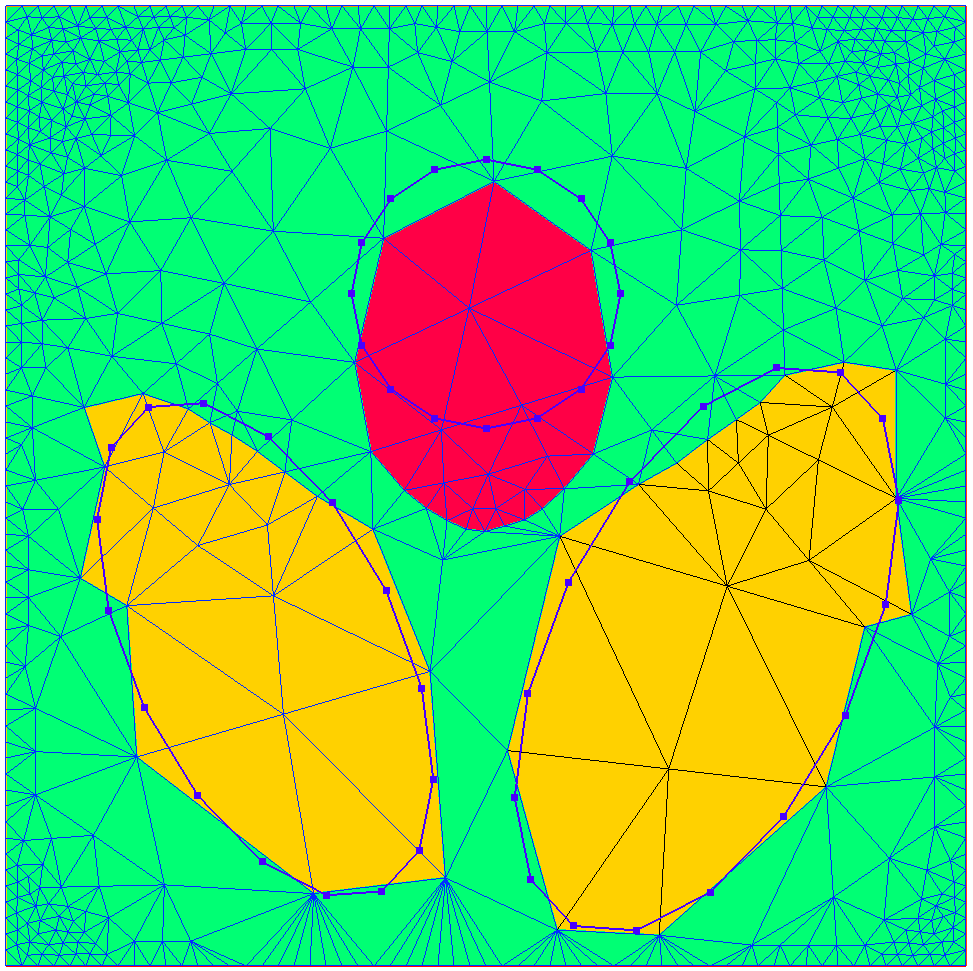}}
\put(170,0){\includegraphics[width=7cm]{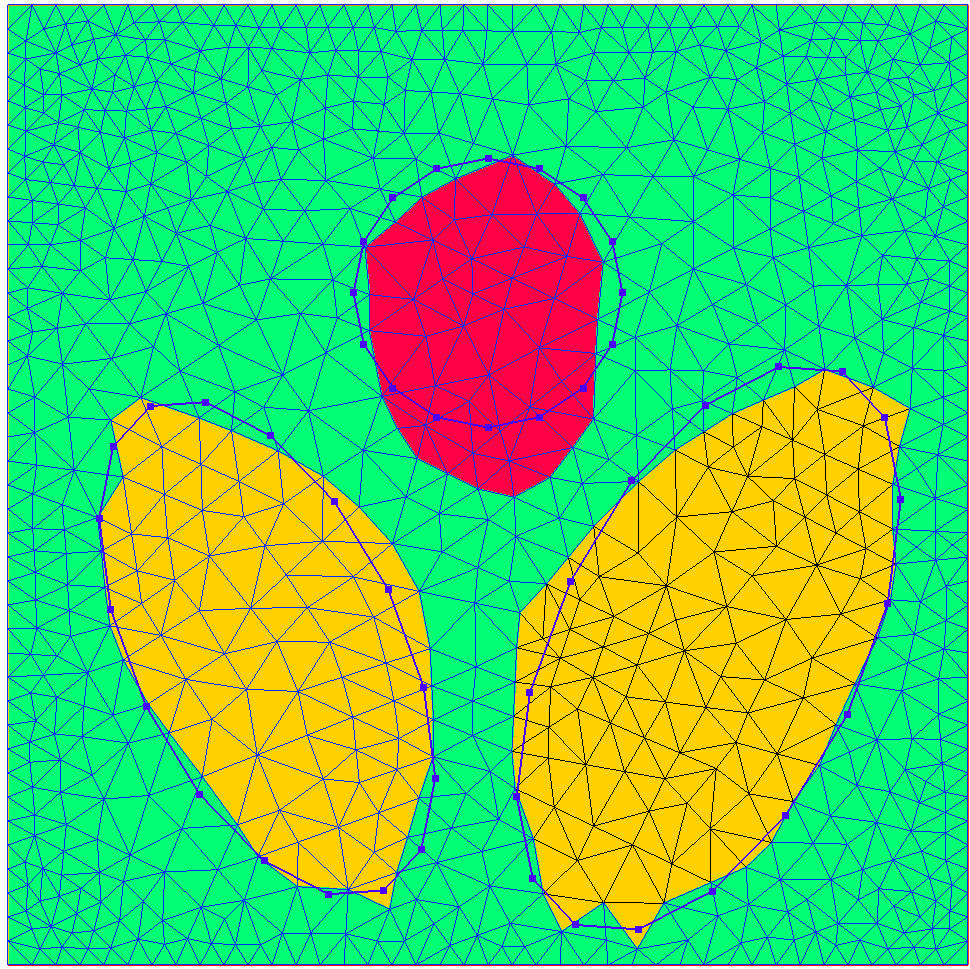}}
\end{picture}
\caption{\label{fig:hnl_6}Reconstruction of the heart and lung phantom starting from 6 measurements. The blue lines highlight the target shapes. The values of the conductivity are known. Initial guess (top-left). Reconstruction from noiseless data with (top-right) and without (bottom-left) the regularization step. Bottom Right: reconstruction with regularization for 3\% noisy data.}
\end{figure}

In this first set of examples, we assume to know the values of the conductivity, and we only reconstruct the shape of the partition. 

In Figure \ref{fig:poly_6}, we consider a piecewise conductivity which is equal to 10 on a convex pentagon and to 1 in the background. The reconstruction is carried out using 6 boundary measurements. The initial guess is a regular polygon of 14 sides (top-left). We present a noiseless reconstruction, with (top-right) and skipping (bottom-left) the regularization step, and a reconstruction with regularization when 3\% of noise is added to the data (bottom-right). The regularization parameters are set to $\delta_1 = 0.7\, \delta$ and $\delta_2 = 1.8\, \delta$, where $\delta$ is the length of the side of the initial guess. Figure \ref{fig:poly_6} shows that the three reconstructions well identify the shape with a comparable precision.

A non-convex polygon is considered in Figure \ref{fig:cav_28}, where the target conductivity is 10 inside and 1 in the background. Here, we employ 28 boundary measurements. The initial guess is a regular polygon of 24 sides (top-left). For this phantom, we present a noiseless reconstruction with (top-right) and without (bottom-left) the regularization step, and a reconstruction post regularization in the presence of a 3\% of noise added to the data. The regularization parameters are chosen as $\delta_1 = 0.85\, \delta$ and $\delta_2 = 1.8\, \delta$, where $\delta$ is the length of the side of the initial guess. We notice here that the regularization step helps to reconstruct more precisely the non-convex part of the unknown. It is also interesting to observe that the shape is well identified also in the noisy case.

In Figure \ref{fig:hnl_6}, we consider the so-called heart and lung phantom \cite{Mueller2002}, with background conductivity 1, two ellipses with conductivity 0.5 and a disk with conductivity 2 (both ellipses and the disk are approximated with 16-sided polygons). The initial guess coincides with three identical regular polygons of 16 sides each (top-left). The regularization parameters are $\delta_1 = 0.9\, \delta$ and $\delta_2 = 1.8\, \delta$, where $\delta$ is the length of the side of one of the initial guess polygons. The reconstruction is done using 6 boundary measurements. Here, the difference between the regularized (top-right) and the non-regularized (bottom-left) reconstruction is significative. Moreover, the reconstruction is very robust to the 3\% noise added to the data (bottom-right).

\begin{rem}
The three phantoms just presented have different contrast. It is therefore natural to study the dependence of the reconstruction on the contrast. We observed that the algorithm converges faster in case of higher contrast, yet the reconstruction quality is the same. For this reason we decided to not include reconstructions of the same phantom with different contrast values.
\end{rem}

\subsection{Reconstruction from a misplaced initial guess}
\begin{figure}
\begin{picture}(300,420)
\put(-60,220){\includegraphics[width=7cm]{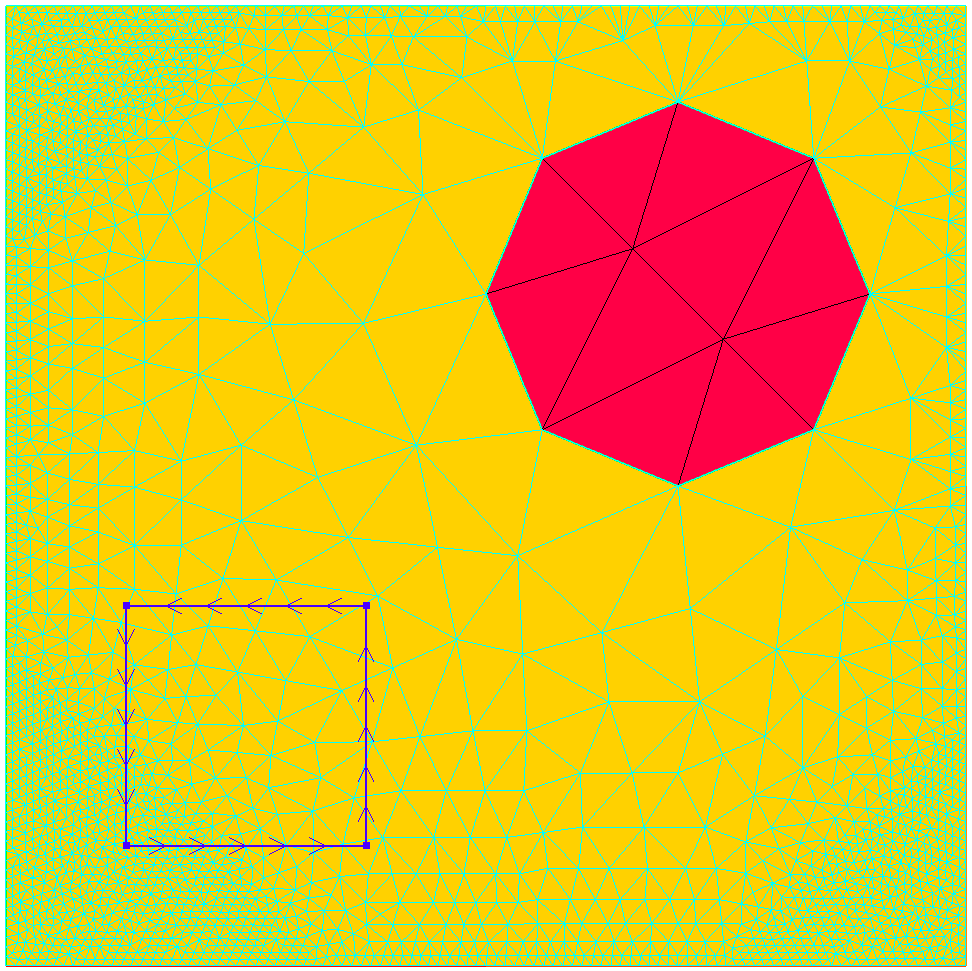}}
\put(170,220){\includegraphics[width=7cm]{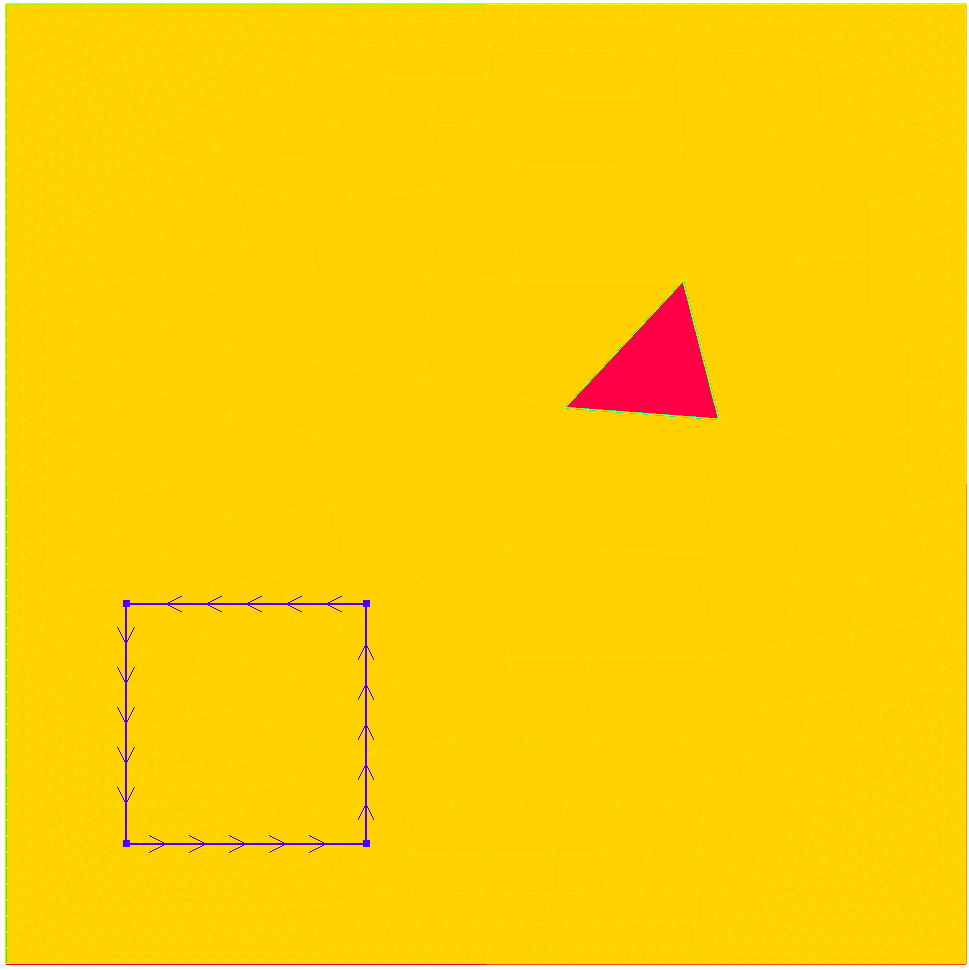}}
\put(-60,0){\includegraphics[width=7cm]{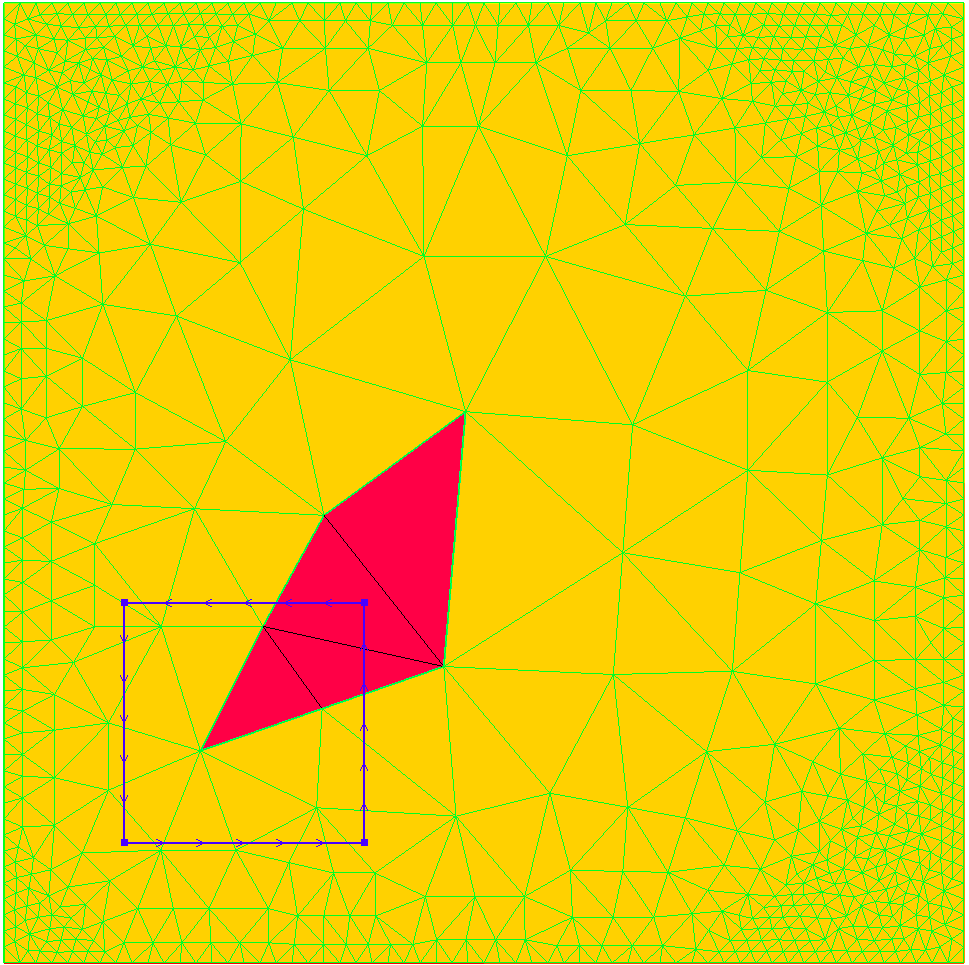}}
\put(170,0){\includegraphics[width=7cm]{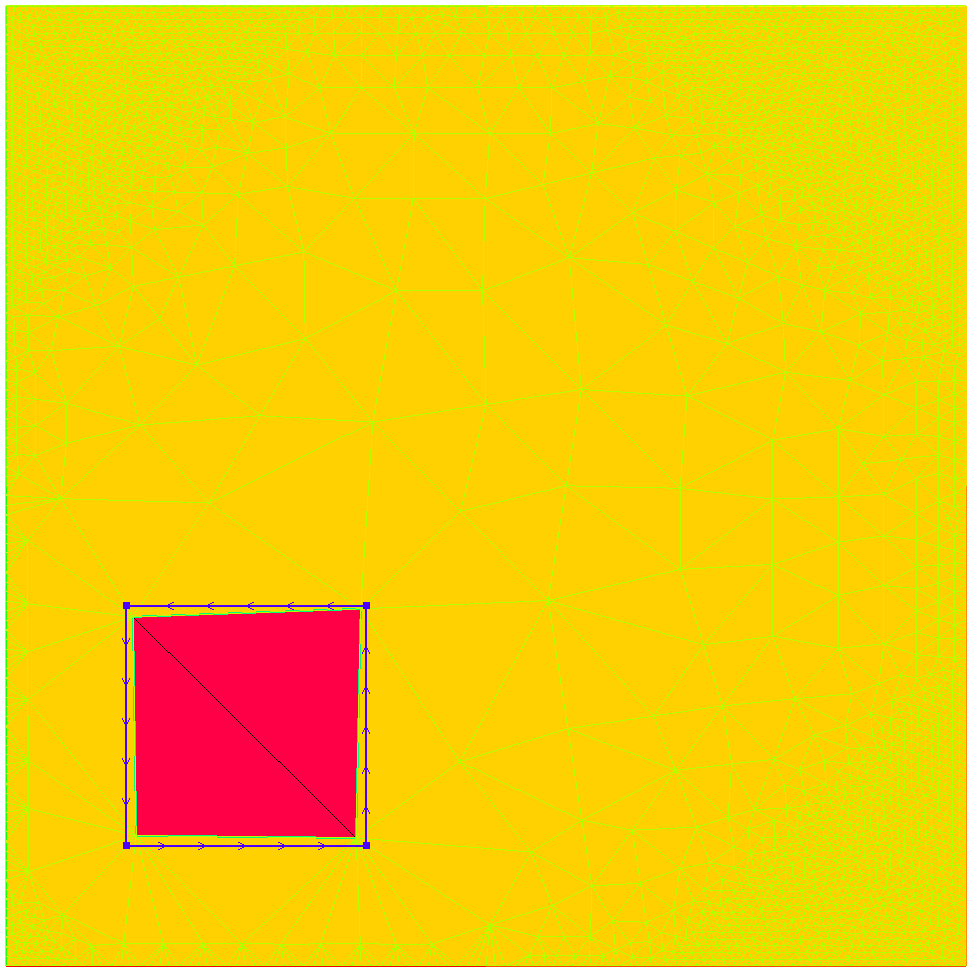}}
\end{picture}
\caption{\label{fig:hnl_far} Reconstruction of a piecewise conductivity on a square from 28 measurements. The blue line represents the target shape. The values of the conductivity are known. Initial guess (top-left). Intermediate configurations (top-right and bottom-left). Final reconstructed conductivity (bottom-right).}
\end{figure}
With the example in Figure \ref{fig:hnl_far}, we show the robustness of the algorithm in recovering an unknown shape from a misplaced initial guess. The target is a piecewise conductivity with value 10 inside a square and 1 in the background. The initial guess is an octagon, far from the exact square
(top-left). The regularization parameters are set to $\delta_1 = 0.8\, \delta$ and $\delta_2 = 1.7\, \delta$, being $\delta$ the length of the side of the octagon. The algorithm is able to recover the square (bottom-right). The striking result strongly depends on the choice of the regularization parameters. Indeed, after the first few iterations, the octagon becomes a triangle (top-right). This reduction in the degrees of freedom avoids potential degeneracy of the shape. The triangle eventually changes and becomes a square when approaching the target (bottom-left).

\subsection{Simultaneous reconstruction}
\begin{figure}
\begin{picture}(300,420)
\put(-60,220){\includegraphics[width=7cm]{images/hnl_0.png}}
\put(170,220){\includegraphics[width=7cm]{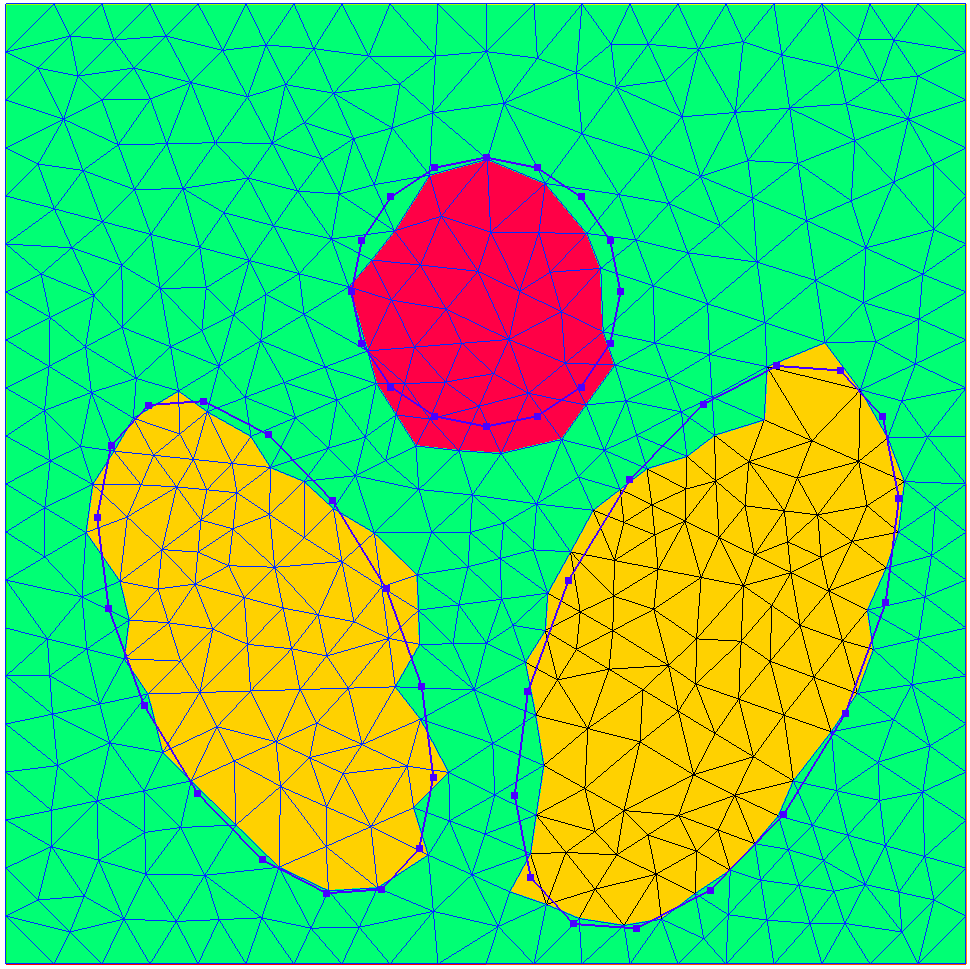}}
\put(-60,0){\includegraphics[width=7cm]{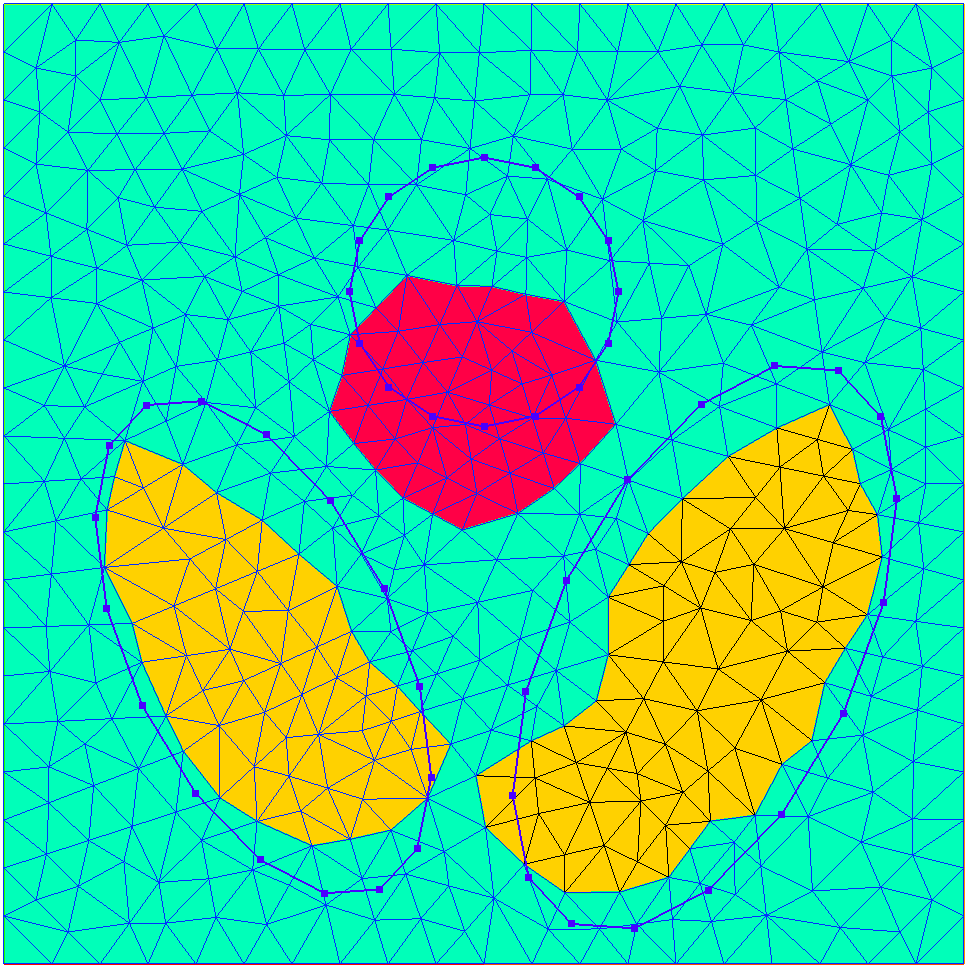}}
\put(170,0){\includegraphics[width=7cm]{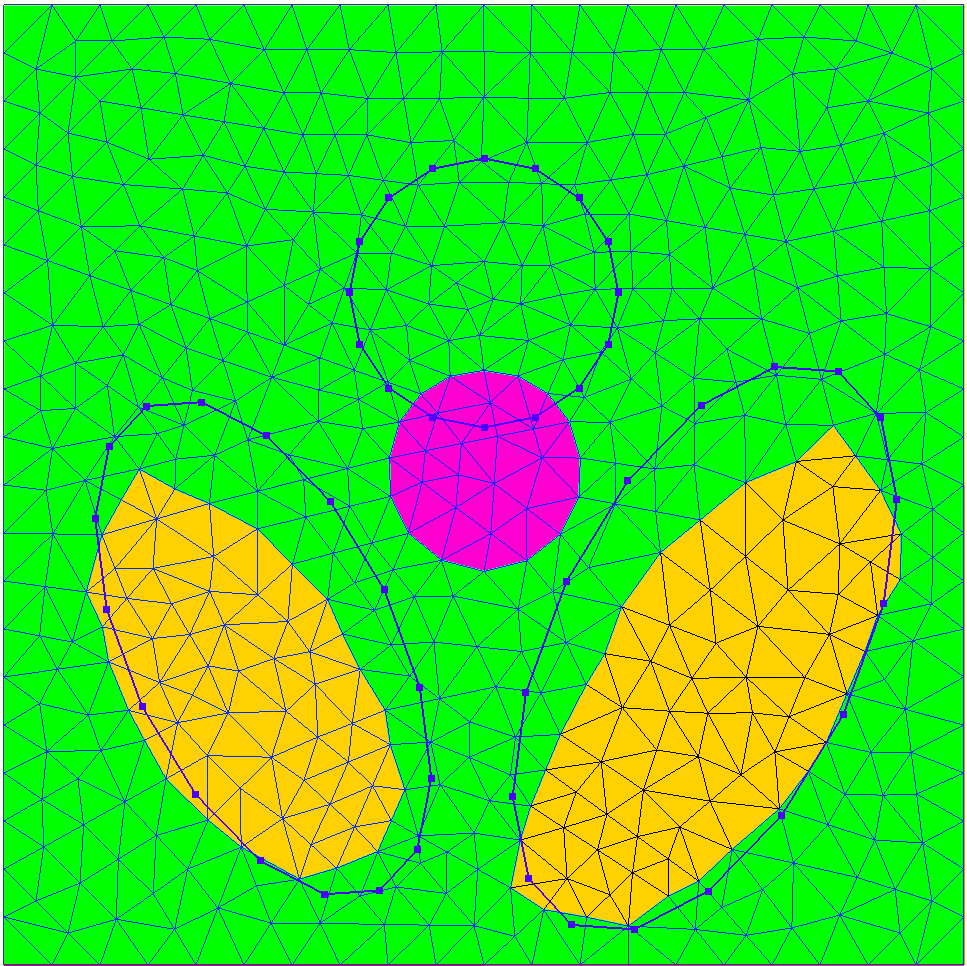}}
\end{picture}
\caption{\label{fig:hnl_28_full}
Reconstruction of the heart and lung phantom starting from 28 measurements. The blue lines highlight the target shapes. The values of the conductivity are unknown. Shape of the initial guess (top-left). Reconstruction from noiseless (top-right) and 5\% noisy (bottom-left) data. Blind reconstruction from 1\% noisy data (bottom-right).
}
\end{figure}
In this section, we present reconstructions obtained with the full algorithm, i.e., the conductivity values are also unknown and updated at each iteration.

We focus on the heart and lung phantom, this configuration being characterized by more interesting features compared with the other ones. The background is assumed to be known and is equal to 1, while the coefficients to be recovered are the lungs, with value 0.5, and the heart, with value 2. The reconstructions presented in Figure \ref{fig:hnl_28_full} are obtained from 28 boundary data. The initial guess (top-left) identifies the shapes the algorithm acts on, while three different values are adopted for the initial conductivity during the reconstruction procedure. The regularization parameters are set to $\delta_1 = 0.9\, \delta$ and $\delta_2 = 1.8\, \delta$ independently of the run, with $\delta$ the length of the side of one of the initial guess polygons.

The first reconstruction deals with noiseless data. The values of the initial guess are 0.55 in the lungs and 2.05 in the heart. The reconstructed values are 0.49 and 2.05, respectively, and the shape is well reconstructed for the three inclusions (top-right). The second reconstruction starts from the same initial guess, by adding a 5\% noise to the data. The reconstructed conductivity values are 0.37 in the lungs and 2.06 in the heart. The shape of the lungs is better recovered than the heart (bottom-left).

The last reconstruction is essentially \textit{blind}, where we assume that the initial guess is provided by a constant conductivity equal to 1 (the background). Moreover, data contains 1\% of additive noise. Despite the challenging setting, some features are correctly recovered. The reconstructed values are 0.44 in the lungs and 0.94 in the heart. The shape and the values of the lungs are well reconstructed, whereas this is not the case for the heart  (bottom-right), due to the lack of a priori information and the central position of the middle polygon in the initial guess.

\subsection{Sensitivity to the number of measurements}
\begin{figure}
\begin{picture}(300,420)
\put(-60,220){\includegraphics[width=7cm]{images/hnl_0.png}}
\put(170,220){\includegraphics[width=7cm]{images/hnl_6.png}}
\put(-60,0){\includegraphics[width=7cm]{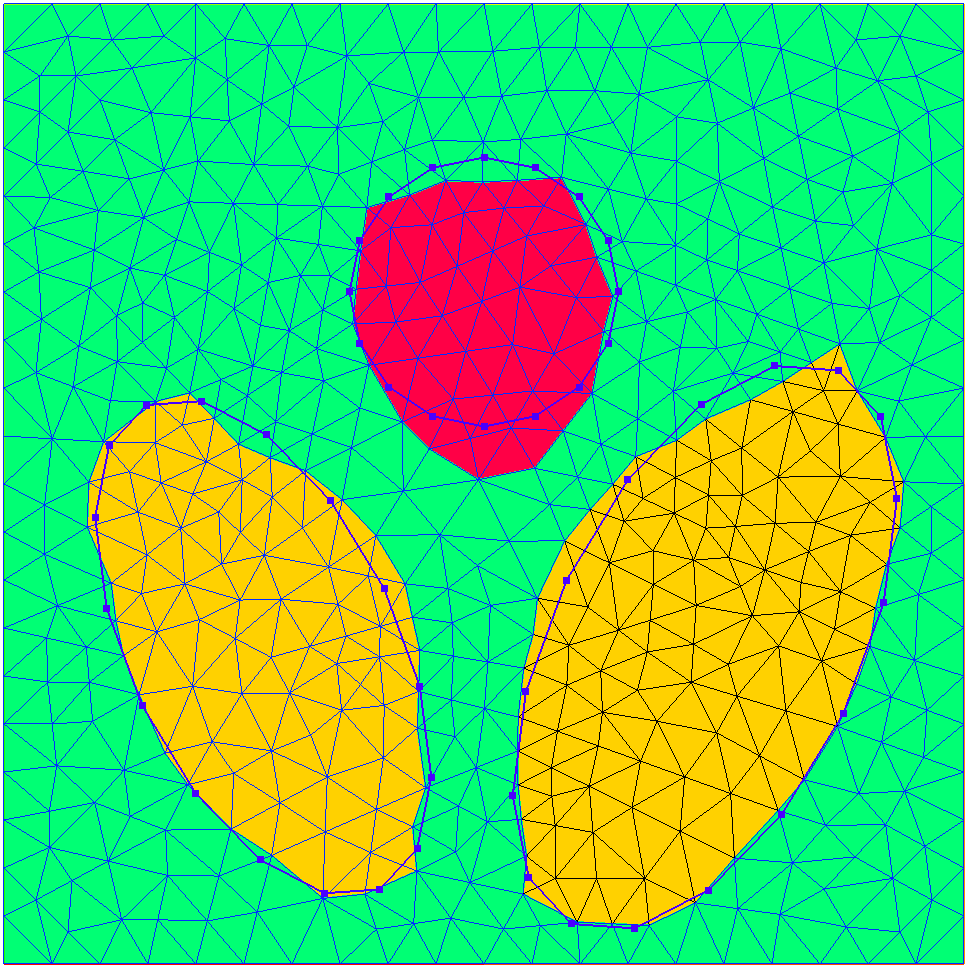}}
\put(170,0){\includegraphics[width=7cm]{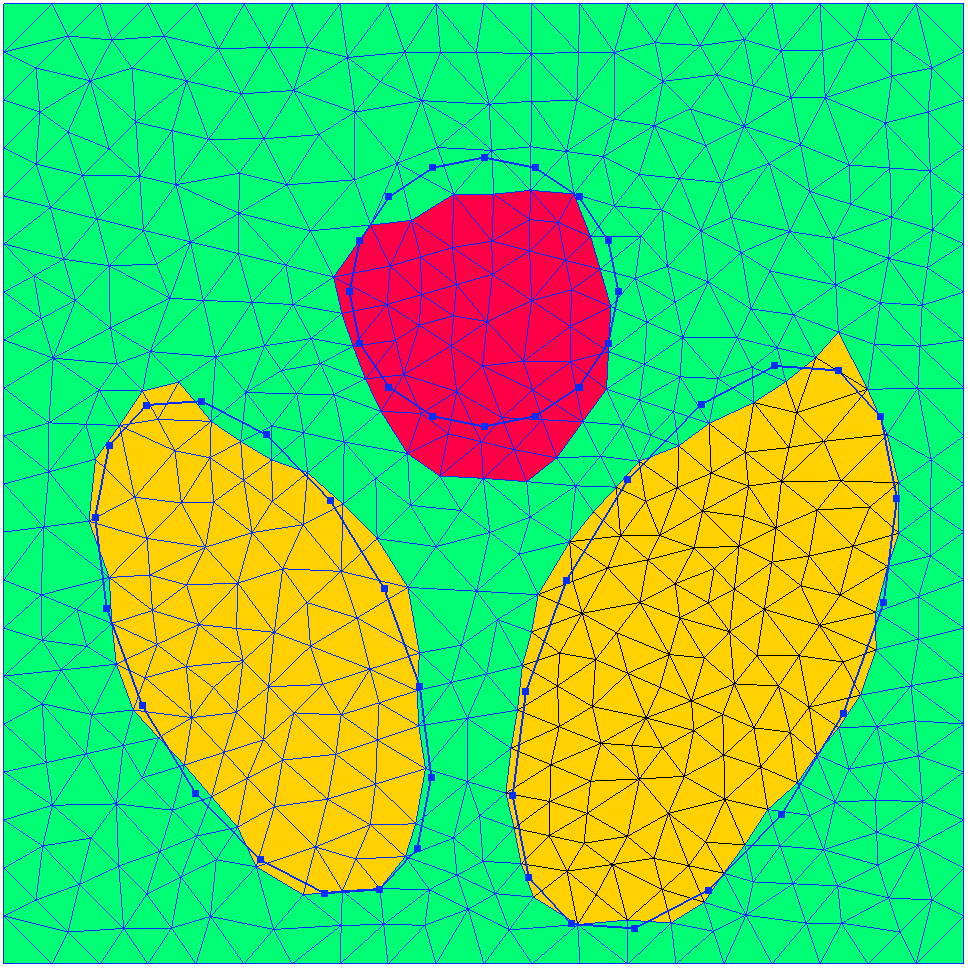}}
\end{picture}
\caption{\label{fig:hnl_all}Reconstruction of the heart and lung phantom starting from 6 (top-right), 28 (bottom-left) and 120 (bottom-right) measurements. The blue lines highlight the target shapes. The values of the conductivity are known. Shape of the initial guess (top-left).}
\end{figure}

Using the heart and lung phantom (always with value 0.5 in the lungs, 2 in the heart, and 1 in the background), we check now how the reconstructions change using a different set of measurements. The regularization parameters are always chosen as $\delta_1 = 0.9\, \delta$ and $\delta_2 = 1.8\, \delta$, with $\delta$ the length of the side of one of the initial guess polygons.

In Figure \ref{fig:hnl_all}, we present shape reconstructions (conductivity values are known) starting from noiseless data sets of 6, 28 and 120 boundary measurements. While there is an evident improvement passing from 6 to 28 data, the reconstruction quality obtained from 28 and 120 measurements looks very similar in this example. This is due to the ill-posedness of the problem that limits the resolution.

\subsection{Sensitivity to the noise level}
\begin{figure}
\begin{picture}(300,420)
\put(-60,220){\includegraphics[width=7cm]{images/hnl_0.png}}
\put(170,220){\includegraphics[width=7cm]{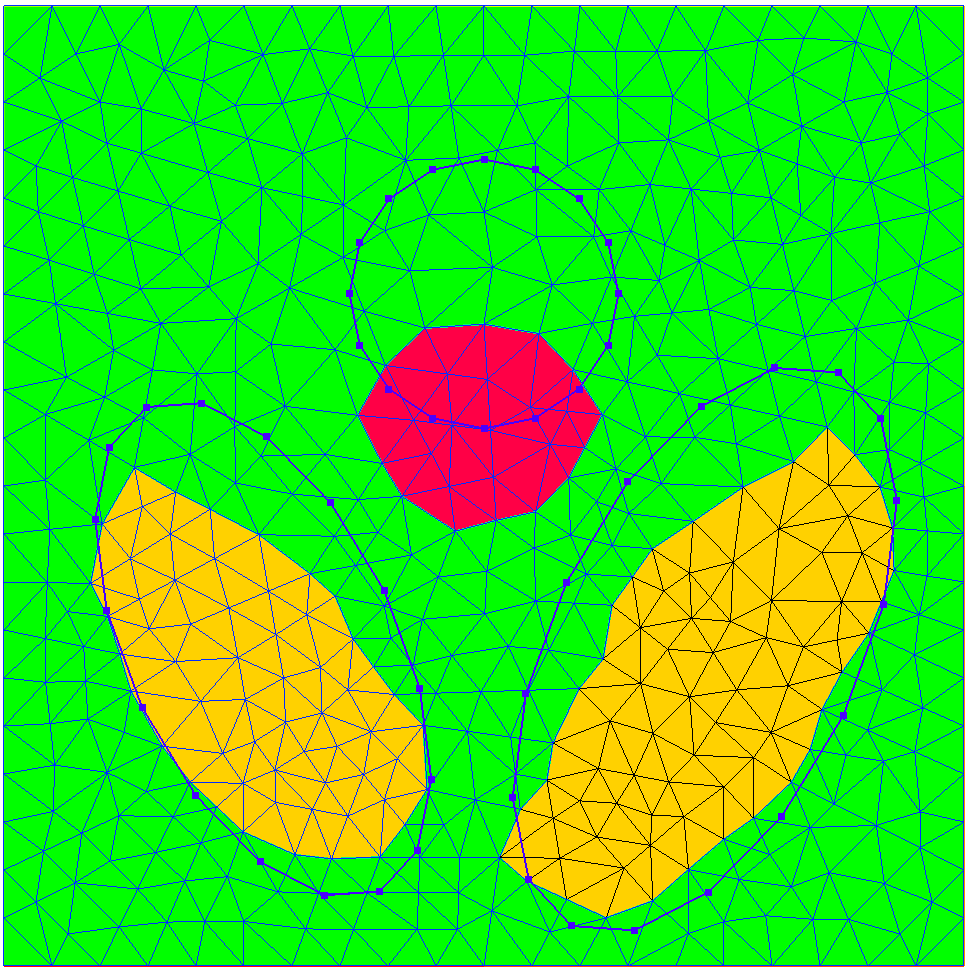}}
\put(-60,0){\includegraphics[width=7cm]{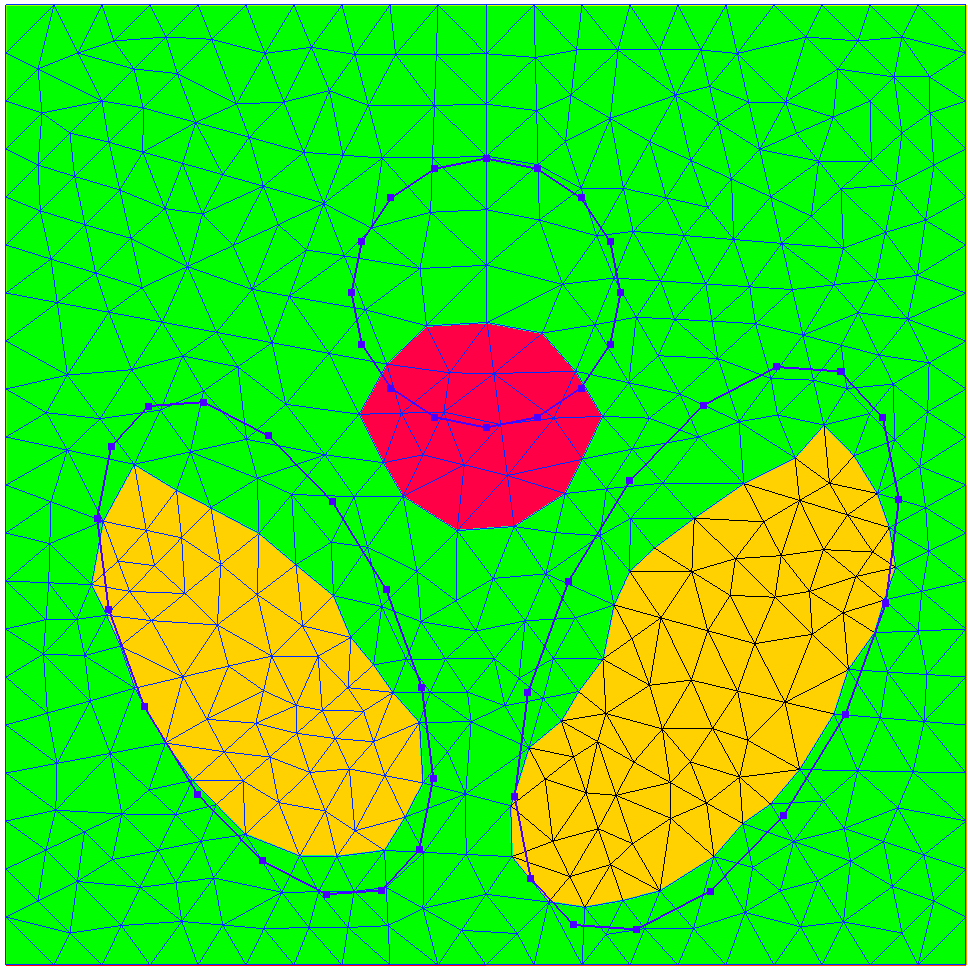}}
\put(170,0){\includegraphics[width=7cm]{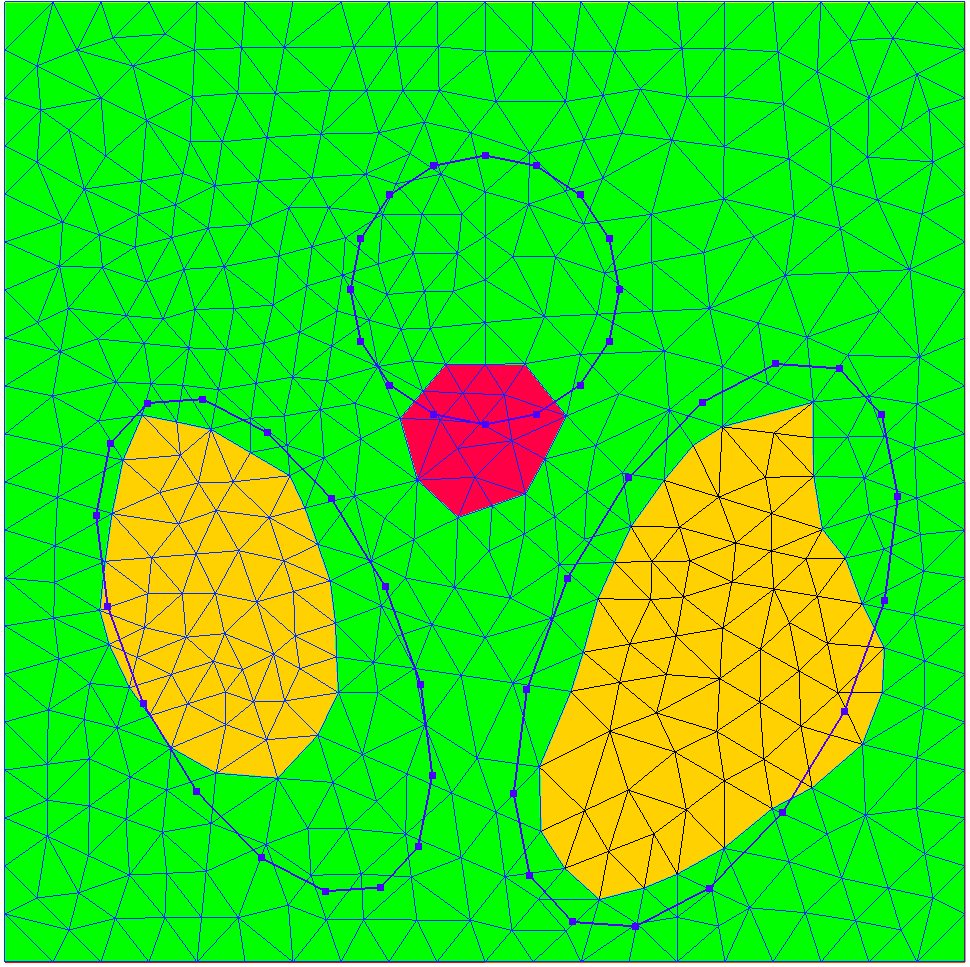}}
\end{picture}
\caption{\label{fig:hnl_noise}Reconstruction of the heart and lung phantom starting from noisy measurements. The blue lines highlight the target shapes. The values of the conductivity are unknown. Shape of the initial guess (top-left). Reconstruction starting from 0.5\% (top-right), 5\% (bottom-left) and 20\% (bottom-right) noisy data.}
\end{figure}

In this section, we show how the algorithm is stable to noise in the measurements. In Figure \ref{fig:hnl_noise}, we present reconstructions using the full algorithm (unknown shape and values) for the heart and lung phantom. The regularization parameters are always chosen as $\delta_1 = 0.9\, \delta$ and $\delta_2 = 1.8\, \delta$, being $\delta$ the length of the side of one of the initial guess polygons.

The initial guess has values 0.7 for the lungs and 1.5 for the heart, far from the exact values 0.5 and 2, respectively. Data contain 0.5\% (top-right), 5\% (bottom-left) and 20\% (bottom-right) of additive noise, respectively. The tolerance in the stopping criterion is chosen experimentally, according to the noise level: $0.004$ for both 0.5\% and 5\% of noise, $0.02$ for 20\% of noise. 
We observe that the shape of the lungs is well recovered even with very noisy data, whereas the heart is poorly recovered, due to the values of the initial guess and the central position of the middle disk. The reconstructed values exhibit a low sensitivity to the noise level. They are 0.4 in the lungs and 1.51 in the heart for both 0.5\% and 5\% of noise, while we obtain 0.37 in the lungs and 1.49 in the heart when considering 20\% of noise.

%

\section{Conclusions}

We have presented a new shape optimization approach that is able to well identify a piecewise constant conductivity on a polytopal partition in electrical impedance tomography. Despite the ill-posedness of EIT, the a-priori assumption on the conductivity to be piecewise constant on a polygonal partition regularizes the problem. This fact, coupled with the effectiveness of the algorithm, allows us to recover an unknown partition (or at least a subset of it) even in the case of noisy data and a wrong initial guess. This is possible due to the use of a distributed shape derivative, to a new regularization step and a new computation of the descent direction.  The algorithm performs better when a good approximation of the values of the conductivity is available. When the coefficients are not known, we noticed that it takes more iterations to approximate their values than to identify the shapes. These values could be obtained, for instance, from a one-step reconstruction \cite{Siltanen2000} and used as a first guess to recover the interfaces more accurately (see also \cite{Hamilton2016} for a hybrid one-step method). The algorithm can be accelerated using more advanced minimization methods and more efficient forward solvers. Possible approaches to accelerate the minimization include Newton-type methods and also reduced order models to speed up the forward and adjoint problems. This would be crucial with a view to practical applications.

\section*{Acknowledgments}
E. Beretta and M. Santacesaria thank the New York University Abu Dhabi for its kind hospitality that permitted a further development of the present research. The work of E. Beretta and M. Santacesaria was partially supported by GNAMPA (Gruppo Nazionale per l'Analisi Matematica, la Probabilità e le loro Applicazioni).

\bibliographystyle{elsarticle-num}

\begin{thebibliography}{10}
\expandafter\ifx\csname url\endcsname\relax
  \def\url#1{\texttt{#1}}\fi
\expandafter\ifx\csname urlprefix\endcsname\relax\def\urlprefix{URL }\fi
\expandafter\ifx\csname href\endcsname\relax
  \def\href#1#2{#2} \def\path#1{#1}\fi

\bibitem{Calder'on1980}
A.-P. {C}alder{\'o}n, On an inverse boundary value problem, in: Seminar on
  {N}umerical {A}nalysis and its {A}pplications to {C}ontinuum {P}hysics ({R}io
  de {J}aneiro, 1980), Soc. Brasil. Mat., Rio de Janeiro, 1980, pp. 65--73.

\bibitem{Chambers2006}
J.~E. Chambers, O.~Kuras, P.~I. Meldrum, R.~D. Ogilvy, J.~Hollands, Electrical
  resistivity tomography applied to geologic, hydrogeologic, and engineering
  investigations at a former waste-disposal site, Geophysics 71~(6) (2006)
  B231--B239.

\bibitem{Karhunen2010}
K.~Karhunen, A.~Sepp{\"a}nen, A.~Lehikoinen, P.~J.~M. Monteiro, J.~P. Kaipio,
  Electrical resistance tomography imaging of concrete, Cement and Concrete
  Research 40 (2010) 137--145.

\bibitem{Sylvester1987}
J.~Sylvester, G.~Uhlmann, A global uniqueness theorem for an inverse boundary
  value problem, Annals of Mathematics 125 (1987) 153--169.

\bibitem{Novikov1988}
R.~G. Novikov, A multidimensional inverse spectral problem for the equation
  $-{\Delta} \psi+(v(x)-{E}u(x))\psi = 0$, Functional Analysis and Its
  Applications 22~(4) (1988) 263--272.

\bibitem{Nachman1996}
A.~I. Nachman, Global uniqueness for a two-dimensional inverse boundary value
  problem, Annals of Mathematics 143 (1996) 71--96.

\bibitem{Alessandrini1988}
G.~Alessandrini, Stable determination of conductivity by boundary measurements,
  Applicable Analysis 27 (1988) 153--172.

\bibitem{mandache2001}
N.~Mandache, Exponential instability in an inverse problem for the
  {S}chr{\"o}dinger equation, Inverse Problems 17~(5) (2001) 1435.

\bibitem{Cheney1990a}
M.~Cheney, D.~Isaacson, J.~Newell, S.~Simske, J.~Goble, {NOSER}: An algorithm
  for solving the inverse conductivity problem, International Journal of
  Imaging Systems and Technology 2~(2) (1990) 66--75.

\bibitem{Rondi2001}
L.~Rondi, F.~Santosa, \href{http://dx.doi.org/10.1051/cocv:2001121}{Enhanced
  electrical impedance tomography via the {M}umford-{S}hah functional}, ESAIM
  Control Optim. Calc. Var. 6 (2001) 517--538.
\newblock \href {http://dx.doi.org/10.1051/cocv:2001121}
  {\path{doi:10.1051/cocv:2001121}}.
\newline\urlprefix\url{http://dx.doi.org/10.1051/cocv:2001121}

\bibitem{Chung2005electrical}
E.~T. Chung, T.~F. Chan, X.-C. Tai, Electrical impedance tomography using level
  set representation and total variational regularization, Journal of
  Computational Physics 205~(1) (2005) 357--372.

\bibitem{Chan2004}
T.~F. Chan, X.-C. Tai, Level set and total variation regularization for
  elliptic inverse problems with discontinuous coefficients, Journal of
  Computational Physics 193~(1) (2004) 40--66.

\bibitem{Chen1999}
Z.~Chen, J.~Zou, An augmented {L}agrangian method for identifying discontinuous
  parameters in elliptic systems, SIAM Journal on Control and Optimization
  37~(3) (1999) 892--910.

\bibitem{Bruhl2000}
M.~Br\"uhl, M.~Hanke, Numerical implementation of two non-iterative methods for
  locating inclusions by impedance tomography, Inverse Problems 16 (2000)
  1029--1042.

\bibitem{Kirsch2008}
A.~Kirsch, N.~Grinberg, The Factorization Method for Inverse Problems, Oxford
  University Press, USA, 2008.

\bibitem{Siltanen2000}
S.~Siltanen, J.~Mueller, D.~Isaacson, An implementation of the reconstruction
  algorithm of {A}. {N}achman for the 2-{D} inverse conductivity problem,
  Inverse Problems 16 (2000) 681--699.

\bibitem{Ikehata2000a}
M.~Ikehata, S.~Siltanen, Numerical method for finding the convex hull of an
  inclusion in conductivity from boundary measurements, Inverse Problems 16~(4)
  (2000) 1043--1052.

\bibitem{Tamburrino2002}
A.~Tamburrino, G.~Rubinacci, A new non-iterative inversion method for
  electrical resistance tomography, Inverse Problems 18~(6) (2002) 1809.

\bibitem{Kaipio2004a}
J.~Kaipio, E.~Somersalo, Statistical and Computational Inverse Problems, Vol.
  160 of Applied Mathematical Sciences, Springer Verlag, 2004.

\bibitem{Alessandrini2005a}
G.~Alessandrini, S.~Vessella, Lipschitz stability for the inverse conductivity
  problem, Advances in Applied Mathematics 35~(2) (2005) 207--241.

\bibitem{beretta2015stable}
E.~Beretta, M.~V. de~Hoop, E.~Francini, S.~Vessella, Stable determination of
  polyhedral interfaces from boundary data for the {H}elmholtz equation,
  Communications in Partial Differential Equations 40~(7) (2015) 1365--1392.

\bibitem{beretta2016stable}
E.~Beretta, M.~V. de~Hoop, F.~Faucher, O.~Scherzer, Inverse boundary value
  problem for the {H}elmholtz equation: quantitative conditional {L}ipschitz
  stability estimates., SIAM Journal on Mathematical Analysis 48~(6) (2016)
  3962--3983.

\bibitem{hettlich1998determination}
F.~Hettlich, W.~Rundell, The determination of a discontinuity in a conductivity
  from a single boundary measurement, Inverse Problems 14~(1) (1998) 67.

\bibitem{afraites2007shape}
L.~Afraites, M.~Dambrine, D.~Kateb, Shape methods for the transmission problem
  with a single measurement, Numerical Functional Analysis and Optimization
  28~(5-6) (2007) 519--551.

\bibitem{pantz2005sensibilite}
O.~Pantz, Sensibilit{\'e} de l'{\'e}quation de la chaleur aux sauts de
  conductivit{\'e}, Comptes Rendus Mathematique 341~(5) (2005) 333--337.

\bibitem{beretta2017differentiability}
E.~Beretta, E.~Francini, S.~Vessella, Differentiability of the {D}irichlet to
  {N}eumann map under movements of polygonal inclusions with an application to
  shape optimization, SIAM Journal on Mathematical Analysis 49~(2) (2017)
  756--776.

\bibitem{piccinini1972}
L.~C. Piccinini, S.~Spagnolo, On the {H}{\"o}lder continuity of solutions of
  second order elliptic equations in two variables, Annali della Scuola Normale
  Superiore di Pisa-Classe di Scienze 26~(2) (1972) 391--402.

\bibitem{bruce1974}
R.~Bruce~Kellogg, On the {P}oisson equation with intersecting interfaces,
  Applicable Analysis 4~(2) (1974) 101--129.

\bibitem{Laurain2016}
A.~Laurain, K.~Sturm, Distributed shape derivative via averaged adjoint method
  and applications, ESAIM: Mathematical Modelling and Numerical Analysis 50~(4)
  (2016) 1241--1267.

\bibitem{Giacomini2017}
M.~Giacomini, O.~Pantz, K.~Trabelsi, Certified {D}escent {A}lgorithm for shape
  optimization driven by fully-computable a posteriori error estimators, ESAIM:
  Control, Optimisation and Calculus of Variations 23~(3) (2017) 977--1001.

\bibitem{Hiptmair2015}
R.~Hiptmair, A.~Paganini, S.~Sargheini, Comparison of approximate shape
  gradients, BIT Numerical Mathematics 55~(2) (2015) 459--485.

\bibitem{Giacomini2015}
M.~Giacomini, O.~Pantz, K.~Trabelsi, An a posteriori error estimator for shape
  optimization: application to {EIT}, in: Journal of Physics: Conference
  Series, Vol. 657, IOP Publishing, 2015, p. 012004.

\bibitem{Ito2008}
K.~Ito, K.~Kunisch,
  \href{http://epubs.siam.org/doi/abs/10.1137/1.9780898718614}{Lagrange
  Multiplier Approach to Variational Problems and Applications}, Society for
  Industrial and Applied Mathematics, 2008.
\newblock \href {http://dx.doi.org/10.1137/1.9780898718614}
  {\path{doi:10.1137/1.9780898718614}}.
\newline\urlprefix\url{http://epubs.siam.org/doi/abs/10.1137/1.9780898718614}

\bibitem{bellout1992stability}
H.~Bellout, A.~Friedman, V.~Isakov, Stability for an inverse problem in
  potential theory, Transactions of the American Mathematical Society 332~(1)
  (1992) 271--296.

\bibitem{Dogan2007}
G.~Dogan, P.~Morin, R.~H. Nochetto, M.~Verani, Discrete gradient flows for
  shape optimization and applications, Computer Methods in Applied Mechanics
  and Engineering 196~(37) (2007) 3898--3914.

\bibitem{Hintermuller2008}
M.~Hinterm{\"u}ller, A.~Laurain, Electrical impedance tomography: from topology
  to shape., Control \& Cybernetics 37~(4).

\bibitem{Hecht2012}
F.~Hecht, New development in {F}ree{F}em++, Journal of Numerical Mathematics
  20~(3-4) (2012) 251--266.

\bibitem{Mueller2002}
J.~Mueller, S.~Siltanen, D.~Isaacson, A direct reconstruction algorithm for
  electrical impedance tomography, IEEE Transactions on Medical Imaging 21~(6)
  (2002) 555--559.

\bibitem{Hamilton2016}
S.~J. Hamilton, J.~M. Reyes, S.~Siltanen, X.~Zhang, A hybrid segmentation and
  {D}-bar method for electrical impedance tomography, SIAM Journal on Imaging
  Sciences 9~(2) (2016) 770--793.

\end{thebibliography}

\end{document}